\documentclass{amsart}
\usepackage{graphicx}
\usepackage{amssymb,amsmath}

\newtheorem{theorem}{Theorem}[section]
\newtheorem{lemma}[theorem]{Lemma}

\newtheorem{cor}[theorem]{Corollary}

\newtheorem{defn}[theorem]{Definition}

\newcommand{\K}{\kappa}

\newcommand{\PP}{\mathcal{P}}
\newcommand{\TY}{\nabla\mathrm{Y}}
\newcommand{\YT}{\mathrm{Y}\nabla}
\newcommand{\V}{V}
\newcommand{\E}{E}
\newcommand{\dotcup}{\dot{\cup}}
\newcommand{\ddotcup}{\ddot{\cup}}
\newcommand{\ol}{\overline}

\title{Six variations on a theme: almost planar graphs}
\author[Lipton, Mackall, Mattman, Pierce, Robinson, Thomas, Weinschelbaum]{Max Lipton, Eoin Mackall, Thomas W.\ Mattman, Mike Pierce, Samantha Robinson, Jeremy Thomas, and Ilan Weinschelbaum}
\address{
Mathematics Department,
Ford Hall,
Willamette University,
900 State Street,
Salem, Oregon 97301}
\email{mlipton@willamette.edu}
\address{Department of Mathematics and Statistics,
California State University, Chico,
Chico, CA 95929-0525}
\email{eoinmackall@yahoo.com}
\email{TMattman@CSUChico.edu}
\email{mpierce9@mail.csuchico.edu}
\email{jthomas72@mail.csuchico.edu}
\address{Etna High School,
400 Howell Avenue,
P.O. Box 721,
Etna, California 96027 }
\email{mrsrobinsonmath@gmail.com}
\address{
Department of Mathematics and Computer Science,
Wesleyan University,
45 Wyllys Avenue,
Middletown, CT 06459}
\email{iweinschelba@wesleyan.edu}

\thanks{Research supported in part by an NSF REUT grant as well as the provost and math department of CSU, Chico.}

\begin{document}

\begin{abstract}
A graph is \textbf{apex} if it can be made planar by deleting a vertex, that
is, $\exists v$ such that $G-v$ is planar. We define the related notions of \textbf{edge apex},
$\exists e$ such that $G-e$ is planar, and \textbf{contraction apex}, $\exists e$
such that $G/e$ is planar, as well as the analogues with a universal quantifier:
$\forall v$, $G-v$ planar; $\forall e$, $G-e$ planar; and $\forall e$, $G/e$ planar.
The Graph Minor Theorem of Robertson and Seymour ensures that each of these
six gives rise to a finite set of obstruction graphs. For the three
definitions with universal quantifiers we determine this set. 
For the remaining properties, 
apex, edge apex, and contraction apex, we show
there are at least 36, 55, and 82 obstruction graphs respectively. We give
two similar approaches to almost nonplanar ($\exists e$, $G+e$ is nonplanar and
$\forall e$, $G+e$ is nonplanar) and determine the corresponding minor minimal graphs.
\end{abstract}

\maketitle

\section{Introduction}
Kuratowski~\cite{K} showed that the set of planar graphs is determined by
two obstructions.
\begin{theorem} \cite{K,W}
A graph is planar if and only if it has no $K_5$ nor $K_{3,3}$ minor.
\end{theorem}
We give the formulation in terms of minors due to Wagner~\cite{W}
to make the connection with Robertson and Seymour's~\cite{RS}
Graph Minor Theorem.
We say $H$ is a \textbf{minor} of graph $G$ if it can be obtained
by contracting edges in a subgraph of $G$. We can state the Graph Minor Theorem as follows.
\begin{theorem}
\cite{RS}
In any infinite set of graphs, there is a pair such that one is a minor of the other.
\end{theorem}

This has two useful consequences. We say $G$ is \textbf{minor minimal $\PP$}
(or MM$\PP$) if $G$ has property $\PP$ but no proper minor does.

\begin{cor}
For any graph property $\PP$, there is a corresponding finite set of minor minimal
$\PP$ graphs.
\end{cor}

\begin{cor}
Let $\PP$ be a graph property that is closed under taking minors. Then there is a finite set of minor minimal non-$\PP$ graphs $S$ such that for any graph $G$, $G$ satisfies $\PP$ if and only if $G$ has no minor in $S$.
\end{cor}

When $\PP$ is minor closed, we say that $S$ is the \textbf{Kuratowski set} for $\PP$.
For example, $\{ K_5, K_{3,3} \}$ is the Kuratowski set for planarity.

The Graph Minor Theorem is not constructive,
so there are only a few graph properties $\PP$
for which we know the finite set of MM$\PP$ graphs.
In particular, there are several graph properties
closely related to planarity for which
this set is unknown.
Our goal in this paper is 
to investigate the minor minimal sets for the 
following eight graph properties.

\begin{defn}
A planar graph is \textbf{almost nonplanar (AN)} if there exist two nonadjacent vertices such that adding an edge between the vertices yields a nonplanar graph.
A planar graph is \textbf{completely almost nonplanar (CAN)} if it is not complete and adding an edge between any pair of nonadjacent vertices yields a nonplanar graph.
\end{defn}

Let $G-v$ (respectively, $G-e$, $G/e$) denote the graph resulting from deletion of vertex $v$ and its edges (respectively, deletion of edge $e$, contraction of edge $e$) in graph $G$.

\begin{defn} A graph is \textbf{not apex (NA)} if, for all vertices $v$, $G-v$ is nonplanar. Similarly, a graph is \textbf{not edge apex (NE)} if,
for all edges $e$, $G-e$ is nonplanar and \textbf{not contraction apex (NC)} if,
for all edges $e$, $G/e$ is nonplanar.
\end{defn}

\begin{defn} A graph $G$ is \textbf{incompletely apex (IA)} if there is a vertex $v$ such that $G-v$ is nonplanar, \textbf{incompletely edge apex (IE)} if there is an edge $e$ such that $G-e$ is nonplanar, and \textbf{incompletely contraction apex (IC)}  if there is an edge $e$ such that $G/e$ is nonplanar.
\end{defn}

We call these last three properties `incomplete' in contrast to their negations. For example, we think of a graph as `completely' apex if $G-v$ is planar for every vertex $v$. Table~\ref{tbldfn} gives a summary of our eight definitions.

\begin{table}[h]
   \begin{center}
      \begin{tabular}{c|l}
         Property               & Definition  \\ \hline \hline
         AN & $\exists e, G+e$ is nonplanar, where $G$ is planar \\ \hline
         CAN & $\forall e, G+e$ is nonplanar, where $G$ is planar, not complete \\ \hline
         NA    & $\forall v, G-v$ is nonplanar  \\ \hline  
         NE    & $\forall e, G-e$ is nonplanar   \\ \hline
         NC   & $\forall e, G/e$ is nonplanar \\ \hline
         IA      & $\exists v,  G-v$ is nonplanar \\ \hline
         IE  & $\exists e, G-e$ is nonplanar \\ \hline
         IC & $\exists e, G/e$ is nonplanar  \\ 
      \end{tabular}
   \end{center}
\vspace{5pt}
\caption{Comparision of the eight definitions.}
\label{tbldfn}
\end{table}

\begin{table}[h]
    \begin{center}
        \begin{tabular}{l||c|c|c|c|c|c|c|c}
            Graph Property $\PP$     & AN & CAN & NA        & NE        & NC        & IA  & IE  & IC  \\ \hline
            Is (Not $\PP$) Minor Closed?   & No & No  & Yes       & No        & No        & Yes & Yes & Yes \\ \hline

            Number of MM$\PP$ Graphs & 2  & 1   & $\geq 36$ & $\geq 55$ & $\geq 82$ & 2   & 5   & 7   \\
        \end{tabular}
    \end{center}
\vspace{5pt}
\caption{Results for the eight graph properties.}
\label{tblres}
\end{table}

We summarize our results in Table~\ref{tblres}.
Four of the properties give Kuratowski sets (as their negation generates 
a minor closed set) and with the exception
of NA, NE, and NC, we determine the finite set of MM$\PP$ graphs.
For the remaining three 
properties
we give a lower bound, which is
simply the number of MM$\PP$ graphs we have found, so far.

Our paper is organized as follows. Below we conclude this introduction with a
survey of the literature and provide some preliminary notions used throughout the paper. In Section~2 we determine the MMAN and MMCAN graphs and show that
neither is a Kuratowski set. The following section is our classification
of the MMIA, MMIE, and MMIC graphs, all three of which we show are
Kuratowski. In Section~4 we give an overview of the MMNA graphs, a Kuratowski set.
We classify graphs in this family of connectivity at most one.
For graphs of connectivity two, with $\{a,b\}$ a 2--cut, we classify those for which
$ab \in \E(G)$ as well as those for which a component of $G-a,b$ is nonplanar.
We also prove that a MMNA
graph has connectivity at most five.
In total, we give explicit constructions for 36 MMNA graphs.
Finally, in Section~5 we discuss MMNE and MMNC graphs, first showing these are not Kuratowski.
We classify graphs of connectivity at most one in these two families and discuss computer searches, complete
through graphs of order nine or size 19, that yielded 55 MMNE and 82 MMNA graphs.

Apex graphs are well-studied including
results on MMNA graphs in \cite{A,BM,P}.
Note that Pierce~\cite{P} reports on a computer search that yields
61 MMNA graphs, including all graphs through order ten or size 21
and most of the 36 graphs we describe here.
Different authors have used terms like `almost planar' or
`near planar' in various ways. Here is how our
definitions relate to others in the literature.
Cabello and Mohar~\cite{CM} say that
a graph is {\em near-planar} if it can be obtained from a planar graph by adding an edge. This corresponds to our definition of edge apex. Wagner~\cite{W2} defined
{\em nearly planar (Fastpl\"attbare)}, which corresponds to our idea of completely
apex or not IA. Two further notions of almost planar are not directly related
to the properties we have defined. For Gubser~\cite{G}, a graph $G$ is
{\em almost planar} if for every edge $e$, either $G-e$ or $G/e$ is planar.  In characterizing graphs with no $K^{\aleph_0}$, Diestel, Robertson, Seymour, and Thomas say a graph $G$ is {\em nearly planar} if deleting a bounded number of vertices  makes $G$ planar except for a subgraph of bounded linear width sewn onto the unique cuff of $S^2 - 1$, see~\cite[Section 12.4]{D}. Finally, our notion of CAN is also known as {\em maximally planar}, see Diestel~\cite{D}.

We conclude this introductory section with some notation and definitions, as well as a lemma, used throughout. For us, graphs are simple (no loops or double edges) and undirected. We use $\V(G)$ and $\E(G)$ to denote the vertices and edges of a graph. The \textbf{order} of a graph is $| \V(G)|$ and $|\E(G)|$ is its \textbf{size}. We use $\delta(G)$ to denote the \textbf{minimimum degree} of all the vertices in $G$.

As mentioned earlier, $G-v$, $G-e$, and $G/e$ denote the result of vertex deletion, edge deletion, and edge contraction, respectively. For $v,w \in \V(G)$, $G-v,w$ is the
result of deleting two vertices and their edges. Similarly, for $e,f \in \E(G)$, we define as $G-e,f$ and $G/e,f$ the result of deleting or contracting two edges.
Note that the order of deletion or contraction is arbitrary.
Contracting an edge may result in a double edge. We will assume that one of the doubled edges is deleted so that $G/e$ is again a simple graph.
We use $G_1 \sqcup G_2$ to denote the disjoint union of two 
graphs and $G_1 \dotcup G_2$ for the union
identified on a single vertex.
Similarly, $G_1 \ddotcup G_2$ denotes the union of two graphs identified on two vertices.

In light of Kuratowski's theorem, we call $K_5$ and $K_{3,3}$ the \textbf{Kuratowski graphs} and also refer to them as minor minimal nonplanar or MMNP. A
\textbf{Kuratowski subgraph} or \textbf{K-subgraph} of $G$ is one homeomorphic
to a Kuratowski graph.
A \textbf{cut set} of graph $G$ is a set $U \subset \V(G)$ such that deleting the
vertices of $U$ and their edges results in a disconnected graph. If $|U| = k$, we
call $U$ a 
$k$--\textbf{cut}.
We say $G$ has \textbf{connectivity} $k$ and write $\K(G) = k$ if $k$ is the largest
integer such that $| \V(G) | > k$ and $G$ has no $l$-cut for $l < k$. In particular,
$\K(K_n)  = n-1$.

We conclude this introduction with a useful lemma. In the case that $\K(G) = 2$,
we have $G-a,b = G_1 \sqcup G_2$ where $\{a,b\}$ is a $2$--cut. We will use
$G'_i$ to denote the induced subgraph on $\V(G_i) \cup \{a,b\}$.

\begin{lemma}
\label{lempath}%
If $G$ is homeomorphic to $K_5$ or $K_{3,3}$ with cut set $\{a,b\}$ such that
$G-a,b=G_1\sqcup G_2$, then one of $G'_1$ and $G'_2$ is an $a$-$b$-path.
\end{lemma}

\begin{proof}
Since, $\kappa(K_5)=4$ and
$\kappa(K_{3,3})=3$, $G$ must be a proper subdivision of a Kuratowski graph and,
since they disconnect the graph,
$a$ and $b$ are vertices on a subdivided edge of the underlying $K_5$ or $K_{3,3}$. This means that one of the components is simply an $a$-$b$-path.
\end{proof}

\section{Almost nonplanar: MMAN and MMCAN graphs}
In this section we classify the MMAN and MMCAN graphs.
Let $K_5-e$ denote the complete graph on five vertices with an edge deleted and
$K_{3,3} - e$ the result of deleting an edge in the complete bipartite graph $K_{3,3}$.
The unique MMCAN graph is $K_5-e$ and there are two MMAN graphs, $K_5-e$ and $K_{3,3}-e$. Neither of these are Kuratowski sets, since, for example, $K_5$ is a nonplanar graph (hence neither AN nor CAN) that
contains the MMAN and MMCAN graph $K_{5}-e$ as a minor.

Our classification of the minor minimal CAN graphs makes use of a theorem due to Mader.

\begin{theorem}
\cite{M}
Any graph with $n$ vertices and at least 3n-5 edges contains a subdivision of $K_5$.
\end{theorem}

In Diestel~\cite{D}, CAN is called {\em maximally planar} and it is proved (Proposition 4.2.8) equivalent to a graph admitting a plane triangulation.

\begin{theorem}
Every plane triangulation with at least $5$ vertices has $K_5 -e$ as a minor
\end{theorem}

\begin{proof}
    Let $G$ be a plane triangulation on at least $5$ vertices.
    By Euler's formula, $|\E(G)| = 3(|\V(G)| -2)$.
    Let $G'$ be a nonplanar graph obtained by adding edge $ab$ to $G$.
    Then $|\E(G')| = |\E(G)| + 1 = 3|\V(G)| - 5$.
    By Mader's Theorem $G'$ has a subgraph $H$ homeomorphic to $K_5$.
    Note that we must have $ab \in \E(H)$, else $H$ would be planar.
    Since $H$ is homeomorphic to $K_5$,
    contracting 
    appropriate
    edges in $H - ab$ will result in $K_5-e$,
    showing that $K_5-e$ is a minor of $G$.
\end{proof}

\begin{cor}
The only MMCAN graph is $K_5 - e$.
\end{cor}

\begin{theorem}
The MMAN graphs are $K_5-e$ and $K_{3,3}-e$
\end{theorem}

\begin{proof}
First note that these two graphs are MMAN. Let $G$ be AN and let $ab$ be the edge that is added to form the nonplanar $G'$. By Kuratowski's theorem $G'$ contains a subdivision $H$ of $K_5$ or $K_{3,3}$ and $ab \in \E(H)$. By contracting edges, $H$ gives $K_5 -e$ or $K_{3,3}-e$ as a minor of $G$. So $G$ is MMAN only if it is one of these two.
\end{proof}

\section{Incomplete properties: MMIA, MMIE, and MMIC graphs}
In this section we classify the MMIA, MMIE, and MMIC graphs.
Note that each is a Kuratowski set since the corresponding `complete' property is minor closed.
For example, in the case of the IA graphs, 
suppose $G$ is not IA and let $H$ be a subgraph of $G$.
Then for any $v \in \V(H)$, $H-v$ is planar since it is a subgraph of the planar graph $G-v$.
Similarly if $G$ is not IA, let $H = G/f$ for some $f \in \E(G)$.
Then for any $v \in \V(H)$, $H-v$ is planar since it is a minor of the planar graph $G-v$.
This shows that the property \emph{not} IA (also known as the completely apex property) is minor closed.
Similar arguments show that
not IE and not IC are also minor closed.

We next show there are exactly two MMIA graphs, $K_1 \sqcup K_5$
and $K_1 \sqcup K_{3,3}$. We begin by classifying the disconnected graphs.

\begin{theorem}
If $G$ is not connected and MMIA, then $G=K_{1}\sqcup G_{2}$
where $G_{2} \in \{ K_{5},K_{3,3} \}$.
\end{theorem}

\begin{proof}
Let $G=G_{1}\sqcup G_{2}$ and be MMIA. Consider the
three cases i) both $G_{1}, G_{2}$ are planar, ii) both $G_{1},G_{2}
$ are  nonplanar, iii) one of $G_{1}$, or $G_{2}$, is planar and
the other is nonplanar.

Suppose both components are planar. Then for all $ v\in \V(G)$,
$G-v$ would be planar as well. This contradicts $G$ being IA.

Suppose both components are nonplanar. For $v \in \V(G_{1})$, the graph $G-v$ is nonplanar since $G_{2}$ has not changed. Then, $G-v$ is also
IA because we can delete another vertex
$w \in \V(G_{1})$, that again leaves a new, nonplanar graph.
Having an IA minor, $G-v$, contradicts that $G$ is MMIA,
and is not possible.

So it must be that $G$ is the disjoint union of both a
nonplanar component, say $G_{2}$, and a planar component, $G_{1}$. Since
$G$ is MMIA, no proper minor of $G$ is IA. Then, taking a
proper minor of either $G_{1}$ or $G_{2}$ (and leaving the other
unchanged) cannot result in an IA graph. However, if $G_{1}$ has at
least two vertices, $G$ will have a proper IA minor. So, $G_{1}$
must be a planar graph of at most one vertex, whose deletion leaves only
the graph $G_{2}$. Therefore, $G_{1}=K_{1}$. Furthermore, $G_{2}$
must be nonplanar, so that $G$ is indeed IA, and any minor of
$G_{2}$ must be planar. So, $G=K_{1}\sqcup G_{2}$ where
$G_{2}\in \{ K_{5}, K_{3,3} \} $.

Therefore, if $G$ is not connected and MMIA, then $G=K_{1}\sqcup G_{2}$
where $G_{2}\in \{ K_{5}, K_{3,3} \} $.
\end{proof}

\begin{theorem}
There are no connected MMIA graphs.
\end{theorem}

\begin{proof}
Suppose instead that $G$ is a connected MMIA
graph. Then there is a vertex, $v$, such that $G-v$ is nonplanar.
However, since $G$ is connected, $v$ must have at least one edge,
$e$. Since when deleting a vertex we also delete all of its edges, $G-e$ must be
a proper, nonplanar minor of $G$. However, deleting $v \in \V(G-e)$
is again nonplanar so that $G-e$ is IA. This contradicts the
property that $G$ is MMIA and therefore cannot happen.
\end{proof}

\begin{cor}
There are two MMIA graphs: $K_{1} \sqcup K_5$ and $K_1 \sqcup K_{3.3}$.
\end{cor}

Next we show there are five MMIE graphs. We begin with the disconnected examples. Note that if $G$ has distinct edges $e, e'$ such that $G-e,e'$ is
nonplanar, then $G$ is not MMIE. Indeed, $G-e$ is an IE proper minor.

\begin{theorem} If $G$ is not connected and MMIE, then $G=K_{2}\sqcup G_{2}$
where $G_{2}\in \{ K_{5}, K_{3,3} \}$.
\end{theorem}

\begin{proof}
Let $G=G_{1}\sqcup G_{2}$. If both $G_{1}$ and
$G_{2}$ were planar, then $G$ is not IE. In
addition, $G_{1}$ and $G_{2}$ cannot both be nonplanar because then
$G$ would have a proper IE minor as in the proof for disconnected MMIA
graphs. So, without loss of generality, $G_{1}$ will be planar and $G_{2}$ will be
nonplanar.

The graph $G_{1}$ must then have one and only one edge. If
$G_{1}$ has more than one edge, then $G-e$, where $e\in \E(G_{1})$, will be
an IE graph as well, since we can still remove another edge in
$G_{1}$ to obtain a nonplanar graph. This contradicts the MMIE property of
$G$. In fact, no proper minor of $G_{1}$ can have an edge. So $G_{1}$ must be connected and therefore it is $K_{2}$. In addition, $G_{2}$
must be MMNP. Otherwise, $G_2$ has a proper
nonplanar minor. The corresponding minor of $G$
would then be IE contradicting $G$ is MMIE. So, $G_{2}\in \{ K_{5}, K_{3,3}\} $.
Therefore, if $G$ is not connected and MMIE, then
$G=K_{2}\sqcup G_{2}$ where $G_{2}\in \{ K_{5}, K_{3,3} \}$.
\end{proof}

Recall that $G_1 \dotcup G_2$ denotes the union of $G_1$ and $G_2$ with one vertex in common.

\begin{theorem}
If $G$ is connected, MMIE, and has a cut vertex, then
$G=K_2 \dotcup G_2$ where $G_2\in \{ K_5, K_{3,3} \}$.
\end{theorem}

\begin{proof}
Let $G$ be a connected MMIE graph such that
$G-v=G_{1} \sqcup G_{2}$. Let $G'_{i}$ denote the induced subgraph on $\V(G_{i}) \cup \{v\}$. If both $G'_{1}$ and $G'_{2}$ are
nonplanar, then $G$ would not be MMIE since, for example,
there are two distinct edges $e,e' \in \E(G'_{2})$
such that $G-e,e'$ contains $G'_{1}$ and
is therefore nonplanar. If both subgraphs were planar, then $G$ would
also be planar and therefore not MMIE. So one of  $G'_{1}$ and
$G'_{2}$ is nonplanar, say $G'_1$, and the other, $G'_2$, is planar.

Then $G'_1$ is, in fact, minor minimal nonplanar (MMNP).
Otherwise, $G'_1$ has a proper nonplanar minor $N$.
We can then take a proper minor $H$ of $G$ containing $N$.
Since it contains $N$, $H$ is nonplanar. If $H$ is formed by deleting the cut vertex,
then there's an edge $e\in \E(G'_{1})$, incident to the cut vertex, such that
$G-e$ is nonplanar. Since removing $e$ does not affect $G'_{2}$,
and since $G$ is connected there must be an edge,
$e' \in \E(G'_{2} )$, such that $G-e,e'$ is nonplanar. But this
contradicts that $G$ is MMIE. If $H$ is not formed by deleting the
cut vertex, then $G'_{2}$ is a subgraph of $H$. Since $G$ is connected, there must
be at least one edge $ e\in \E(G'_{2})$ and deleting it, we have
$N$ as a subgraph of $H-e$. However, this contradicts that $G$ is MMIE since
$H-e$ contains $N$ and is therefore nonplanar.

The planar subgraph, $G'_{2}$, must be $K_{2}$. First note that as $G$ is
connected, $G'_2$ is as well.  If there are two distinct edges
$ e_{1},e_{2}\in \E(G'_{2})$, then, since $ G'_1$ is a subgraph of  $G-e_{1},e_{2}$,
$G-e_{1},e_{2}$ is nonplanar which contradicts $G$ being MMIE.
\end{proof}

\begin{theorem}
If $G$ is MMIE, then there is a unique edge $e$ such that $G-e$ is
nonplanar.
\end{theorem}

\begin{proof}
Assume, for the sake of contradiction, that there are $e,e' \in \E(G)$ such that $e\ne e'$ but $G-e$ and $G-e'$
are nonplanar. If $G-e$ is nonplanar, then there is a subgraph of $
G-e$, $H$, with $e\notin \E(H)$, that has a $K_{5}$ or $K_{3,3}$ minor.
Likewise, if $G-e'$ is nonplanar, then it has a nonplanar subgraph $H'$ with $ e' \notin \E(H')$.
If $H'=H$, then $e'=e$. Otherwise, $G-e,e'$ would be nonplanar,
contradicting that $G$ is MMIE. So $H' \ne H$. If $e\notin H'$,
then $G-e,e'$ contains $H'$ and will be nonplanar contradicting
that $G$ is MMIE.

So, $e\in H'$ and, similarly, $e' \in H$.
If $H$ and $H'$ have empty intersection,
then let $ e_{1},e_{2}\in \E(H')$. This means $G-e_{1},e_{2}$
contains $H$ and is nonplanar. This contradicts that $G$ is MMIE.
So, $H$ and $H'$ have nonempty intersection.
If their intersection is
nonplanar, then removing $e$ and $e'$ will not change this
intersection, and $G$ is not MMIE. If their intersection is planar,
then there must be more than one edge in $H'$ that is not in $H$
besides $e$. But, if $H'$ has more edges besides $e$ that are not
in $H$ it would be possible to remove another edge, $f \neq e$,
without changing $H$. This means that $G-f,e$ is nonplanar, and
contradicts that $G$ is MMIE.

Therefore, if $G$ is MMIE, then
there is a unique edge $e$ such that $G-e$ is nonplanar.
\end{proof}

Recall that a K-subgraph is one homeomorphic to $K_5$ or $K_{3,3}$.

\begin{theorem}
If $G$ is MMIE, then the edge $e$, such that $G-e$ is
nonplanar, is not in a K-subgraph. Furthermore, $G-e$ is $
K_{5}$ or $K_{3,3}$.
\end{theorem}

\begin{proof}
Assume, for the sake of contradiction, that $e$
is in a K-subgraph, $H$. Since no graph homeomorphic
to $K_{5}$ or $K_{3,3}$ is IE, $G-e$ is
planar unless $G-e$ contains some other K-subgraph, $H'$.
However, if $G$ contains two K-subgraphs, $H$ and $H'$ with empty
intersection, then $G-e$ will leave $H'$
unchanged. One could then remove a second edge, $f\in \E(H)$, leaving $
H'$ unchanged so that $G-e,f$ nonplanar. This means that $G$ cannot be
MMIE since $G$ would have an IE minor $G-e$. So, $H$ and $H'$ have non empty
intersection.

But $H\ne H'$ since $e$ cannot be an edge in the only
K-subgraph, otherwise $G-e$ is planar. Next, observe that any proper subgraph of a K-subgraph is
planar. This means that for the K-subgraph, $H'$, with $H\ne
H'$, there must be an edge, $g \ne e$, with $g \in \E(H')$ and $g \notin \E(H)$.
Then $G-g$ contains $H$ and is nonplanar. This
contradicts the uniqueness of the edge $e$ and shows $e$ is not in a 
K-subgraph.

Following the same argument as above, $G$ cannot contain more
than one K-subgraph. Indeed, if there were distinct K-subgraphs
$H$ and $H'$, then either the intersection is empty or it is not, and
we achieve similar contradictions as in the
previous argument. So, $G$ contains exactly one K-subgraph.

Finally, the only possible K-subgraph contained in $G$,
$N$, must contain all edges besides $e$. If not, then
there is an edge $ e'\ne e$ such that $G-e'$ is nonplanar. This contradicts
the uniqueness of $e$.
Also, the K-subgraph, $N$, in $G-e$,
must be either $K_{5}$ or $K_{3,3}$. If not, then $N$ would be
a subdivision of either $K_{5}$ or $K_{3,3}$. But, then there is a proper
minor, $G'$, of $G$, by contracting an edge, $e_{1}\in \E(N)$,
which contains a K-subgraph as well.
Provided $e$ remains as an edge of $G'$,
$G'-e$ is nonplanar, contradicting that $G$ is minor minimal.
On the other hand, if contracting $e_{1}$ removes $e$,
then there must be another edge  $e_2$
incident to $e_{1}$, with $e_{2}\in \E(N)$, such that $e$ is incident to
both $e_{1}$ and $e_{2}$. Since $N$ is a subdivision of $K_{5}$
or $K_{3,3}$ and $G/e_{1}$ is nonplanar, $e_{1}$ and $e_{2}$
must be in a path of $N$ formed by subdividing an edge of
the underlying Kuratowski graph. Since $e$ is incident to
both $e_{1}$ and $e_{2}$, $  $ there exists $ N'$, another 
K-subgraph
of $G$ with $e\in E(N')$. This contradicts that there is only one
K-subgraph of  $G$.

So, if $G$ is MMIE then it is made up of either $K_{5}$ or $
K_{3,3}$ and an edge that is not in this K-subgraph.
\end{proof}

\begin{figure}[htb]
\begin{center}
\includegraphics[scale=0.3]{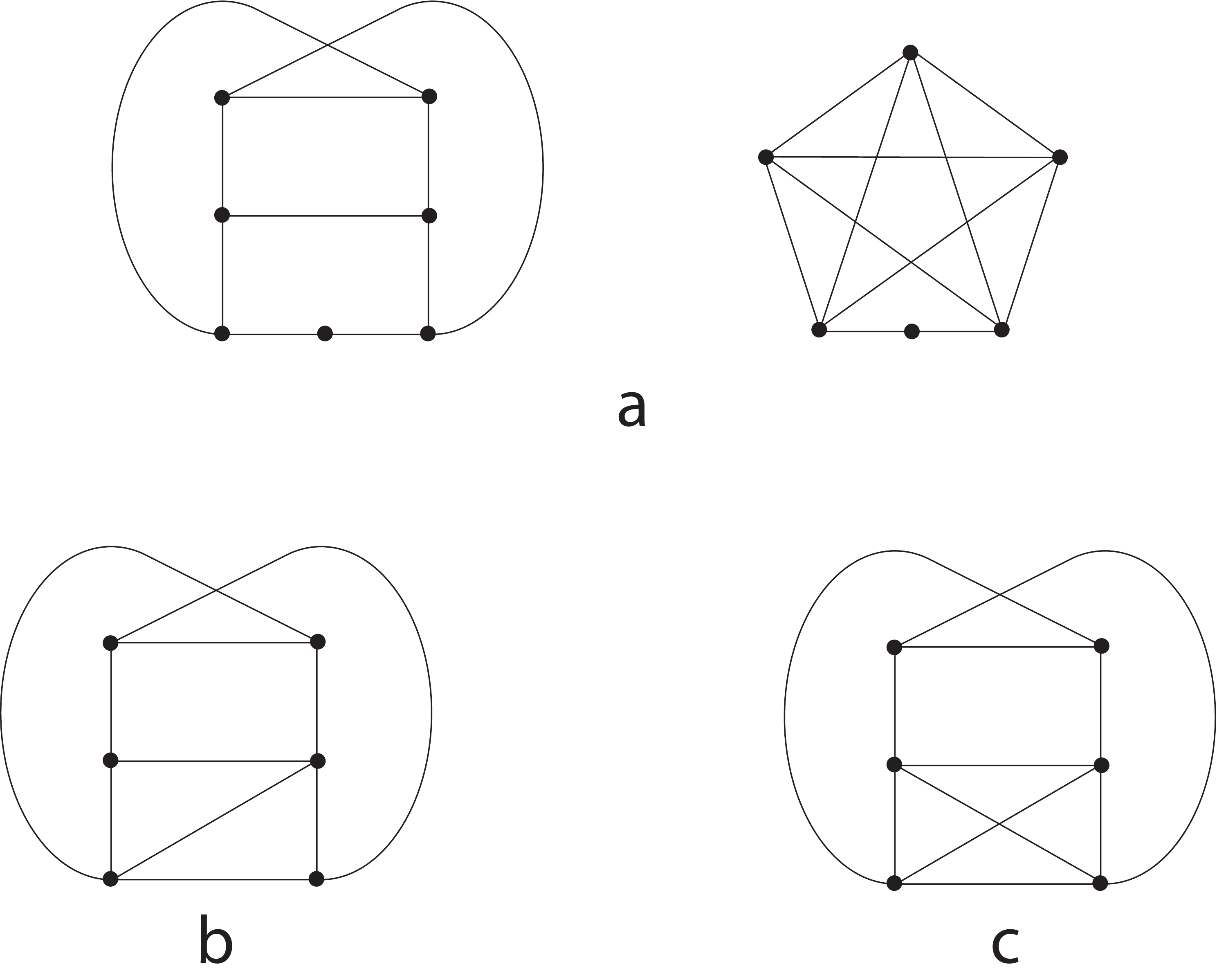}
\caption{MMIE and MMIC graphs.}\label{figMMIEC}
\end{center}
\end{figure}

Aside from the disconnected and connectivity one examples above, a final way to add an edge to a K-subgraph is the graph $K_{3,3}+e$ of Figure~\ref{figMMIEC}b, formed by adding an edge to the bipartite graph $K_{3,3}$.

\begin{cor}
There are five MMIE graphs: $K_{3,3}+e$ and $K_{2}\sqcup G_{2}$, $K_{2} \dot{\cup}  G_2$, where $G_2\in \{K_5, K_{3,3}\}$.
\end{cor}

Let $\ol{K_5}$ (respectively $\ol{K_{3,3}}$) denote the graph obtained from
$K_5$ (resp.\ $K_{3,3}$) by subdividing a single edge,
as in Figure~\ref{figMMIEC}a. We denote as $K_{3,3}+2e$
the graph given by adding two
edges to $K_{3,3}$ as in Figure~\ref{figMMIEC}c.

\begin{theorem}
There are seven MMIC graphs: $K_{3,3} +2e$, $\ol{K}$, $K_2 \sqcup K$, or $K_{2} \dotcup K$ with $K \in \{K_5, K_{3,3}\}$.
\end{theorem}

\begin{proof}
Observe that these seven graphs are MMIC. If $G$ is MMIC and disconnected, then
$G$ is a $K_2 \sqcup K$ with $K$ a Kuratowski graph. We omit the proof which is similar to that for MMIE. Note that the remaining five graphs are precisely the graphs that result when a vertex of a Kuratowski graph is split.

Suppose $G$ is MMIC and connected. Then there is an
edge $e$ such that $G/e$ is nonplanar. Since contracting an edge will not disconnect the graph, $G/e$ is also connected and has a K-subgraph $H$. If $H$ is not a Kuratowski graph, then it has $\ol{K_5}$ or $\ol{K_{3,3}}$ as a minor, contradicting $G$ being minor minimal. Therefore, $H$ is Kuratowski.

If $\V(H) \neq \V(G/e)$, then
since $G/e$ is connected, any vertex in $G/e$ beyond those in $H$, along with one
of its edges  shows that $G/e$
contains $K_2 \sqcup K$ or $K_2 \dotcup K$, with $K$ Kuratowski,
contradicting $G$ minor minimal. So, $\V(H) = \V(G/e)$.

Now $G$ is obtained from $G/e$ by a vertex split. The corresponding vertex split
on $H$ gives rise to a graph $H'$, which is one of the five graphs $K_{3,3}+2e$, $\ol{K}$, or $K_2 \dotcup K$. Since $G$ is minor minimal, then $G = H'$ and is one of these five, hence one of the seven.
\end{proof}

\section{MMNA graphs}
In this section we describe several partial results toward a classification of the MMNA graphs with a focus on graph connectivity. In all, we describe 36 MMNA graphs including all those of connectivity at most one ($\K(G) \leq 1$).
For graphs with $\K(G) = 2$, where 
$\{a,b\}$ 
is a 2--cut, we classify
the MMNA graphs having $ab \in \E(G)$ as well as those for which
a component of $G-a,b$ is nonplanar.
We also
show that $\K(G) \leq 5$ for MMNA graphs, 
which is a sharp bound.
Since the family of apex graphs is minor closed, the MMNA graphs are a Kuratowski set.

We first bound the minimum degree, $\delta(G)$, of an MMNA graph and
then classify the examples with $\K(G) \leq 1$.

\begin{theorem}
\label{thmmmnamindeg3}
    The minimum vertex degree in an MMNA graph is at least three.
\end{theorem}

\begin{proof}
    Let $G$ be an MMNA graph.
    Suppose $G$ has some vertex $v$ of degree zero.
    Since $G$ is MMNA and $G-v$ is a proper minor of $G$,
    there is a  $w  \in \V(G)$ such that $G-v,w$ is planar.
    But the addition of a degree zero vertex to a planar graph is still planar,
    so $(G-v,w)+v = G-w$ is planar, which contradicts $G$ MMNA.

    Next suppose $G$ has some vertex $v$ of degree one.
    Since $G$ is MMNA and $G-v$ is a proper minor of $G$,
    there is a $w \in \V(G)$ such that $G-v,w$ is planar.
    But the addition of a degree one vertex (and its edge) to a planar graph is planar
    (since we can shrink down the edge incident to $v$),
    so $(G-v,w)+v = G-w$ is planar, which contradicts $G$ MMNA.

    Finally suppose $G$ has some vertex $v$ of degree two
    with edges $e$ and $f$ incident to $v$. Contracting $e$
    gives a proper minor $G/e$.
    Since $G$ is MMNA, there is a $w \in \V(G)$ such that $(G/e)-w$ is planar.
    But $f$ remains as an edge in $(G/e)-w$ and subdividing it shows that $G-w$
    is also planar, contradicting $G$ MMNA.
\end{proof}

\begin{theorem} 
\label{thmMMNAdisc}
The disconnected MMNA graphs are $K_5 \sqcup K_5$, $K_5 \sqcup K_{3,3}$, and $K_{3,3} \sqcup K_{3,3}$.
\end{theorem}

\begin{proof} First observe that these three graphs are all MMNA.
On the other hand, if $G = G_1 \sqcup G_2$ is MMNA, both components must
be nonplanar. Otherwise if $G_1$ is planar, then $G_2$ must be NA and is
a proper minor of $G$, contradicting $G$ MMNA. So each component $G_i$ has a
$K_5$ or $K_{3,3}$ minor and $G$ has one of the three candidates as a minor.
Since $G$ is minor minimal, it must be one of the three candidates.
\end{proof}

\begin{theorem} There are no MMNA graphs of connectivity one.
\end{theorem}

\begin{proof} Suppose instead $G$ is MMNA with cut vertex $a$. Then
$G - a = G_1 \sqcup G_2$. If both $G_1$ and $G_2$ are planar, then $G-a$ is planar, contradicting that $G$ is NA. If both are nonplanar, then $G$ has one of the disconnected MMNA graphs as a proper minor and is not minor minimal. So, one
of $G_1$ and $G_2$, say $G_1$, is planar, and the other, $G_2$, is not. Let $G'_i$ denote the induced graph on $\V(G_i) \cup \{a\}$. If $G'_1$ is nonplanar, then
together with $G_2$ this gives one of the three disconnected MMNA graphs as a proper minor of $G$, contradicting $G$'s minor minimality. So $G'_1$ is planar. But then $G'_2$ must be NA, which again contradicts $G$ being minor minimal.
\end{proof}

We can also give an upper bound on the connectivity of an MMNA graph.

\begin{theorem}
If $G$ is MMNA, then $\kappa(G) \leq 5$.
\end{theorem}

\begin{proof}
Suppose $G$ is MMNA and $\K(G) \geq 6$.
For the sake of contradiction, let $D$ be
the largest integer
so that there are two vertices $a,b \in \V(G)$ both of
degree at least $D$.
Note that $D$ is defined. In fact, since $\K(G) \geq 6$, $G$ has minimum degree at least six, so surely $D \geq 6$.
We will argue that there are two vertices with degree at least $D+2$, contradicting
our choice of $D$.
Let $v  = | \V(G) |$ be the number of vertices of $G$.  There will be $v-2$ vertices of degree at least 6 and two vertices of degree at least $D$.  A lower bound on the number of edges of $G$ is then $(6(v-2) + 2D)/2 = 3v-6 + D$.

Since $G$ is MMNA,
we can form a planar graph by deleting an edge (to get a proper minor) and then an apex vertex, which is not adjacent to the deleted edge.  For if it were adjacent to the edge, the vertex deletion would also remove the edge, making $G$ apex, a contradiction.

After deleting an edge,
$G-e$ has at least $3v -7 + D$ edges.
Next delete a vertex, $a \in \V(G)$ of degree $d$.
Then the lower bound on the number of edges in the resulting planar graph is
$3v-7 + D-d$. As this graph is planar on $v-1$ vertices, an upper bound
on the number of edges is $3(v-1) - 6$, the number of edges in a triangulation.
Thus $3v-7 + D-d \leq 3(v-1) - 6$ which implies $d \geq D+2$.

This means $a$'s degree is at least $D+2$.
However, following the argument above, if we first delete an edge incident to $a$, we deduce that there is a second vertex $b$ that is again of degree at least $D+2$.
This is a contradiction since $D$ was assumed to be the maximum such that two vertices have degree at least $D$.
Therefore, if $\K(G) \geq 6$, then $G$ is not MMNA.
\end{proof}

Note that $K_6$ is an MMNA graph of connectivity five, so the bound of the last
theorem is sharp.

The remainder of this section deals with MMNA graphs of connectivity two.
Let us fix some notation for this situation. For $G$ MMNA with cut set $\{a,b\}$, we have $G -a,b = G_1 \sqcup G_2$. Let $G'_i$ denote the induced subgraph on
$\V(G_i) \cup \{a,b\}$.

\begin{theorem}
\label{thmG1G2np}%
Let $G$ be an MMNA graph where $\kappa (G) = 2$, with cut set $\{a, b\}$. If
$G - a,b = G_1 \sqcup G_2$, then $G_1$ and $G_2$ are not both nonplanar.
\end{theorem}

\begin{proof}
Let $c_a$ be an apex of $G-a$. By the assumption that $G$ is MMNA, $G-a,c_a$ is planar. If $c_a = b$, we are done because $G_1 \sqcup G_2 = G - a,b = G - a,c_a$, which would imply both $G_1$ and $G_2$ are planar.

Without loss of generality, assume $c_a \in \V(G_1)$. Since none of the edges of $G_2$ are in $G_1$ and $a,c_a \notin \V(G_2)$, it follows that $G_2$ is a subgraph of the planar graph $G - a,c_a $. Thus, $G_2$ is planar.
\end{proof}

\begin{theorem}
\label{thmgpagpb}%
	If $G$ is MMNA and $\kappa(G) = 2$ such that $G-a,b = G_1 \sqcup G_2$, then, up to relabeling, $G_1+a, G_1+b$ are planar and
	$G_2+a,\,G_2+b$ are nonplanar.
\end{theorem}

We prove this with two lemmas.

\begin{lemma}
	$G_1 +a$ and $G_2+a$ can't both be planar.
\end{lemma}

\begin{proof}
Let $G$ be as described. Suppose both $G_1+a$ and $G_2+a$ are planar.
Since $G_1$ and $G_2$ are otherwise disjoint, $G-b = (G_1+a) \cup (G_2 +a)$ is the union of two planar graphs at only one vertex, with no new edges.
Thus, $G-b$ is planar, which is a contradiction. So it can't be that both $G_1+a$ and $G_2+a$ are planar. A similar argument could be made for $b$.
\end{proof}

\begin{lemma}
$G_1+a$ and $G_2+b$ can't both be nonplanar (up to relabeling).
\end{lemma}

\begin{proof}
Let $G$ be as described. Suppose both $G_1+a$ and $G_2+b$ are nonplanar.
Let $e$ be an edge between a vertex in $G_1$ and the vertex $b$.
Since $G$ is MMNA, $G-e$ is apex. So there is a vertex $v$ such that $(G-e)-v$ is planar.
If $v=a$ then $G_2+b$ is a subgraph of $(G-e)-v$, which is a contradiction since $G_2+b$ is nonplanar.
If $v \in \V(G_1)$ then again $G_2+b$ is a subgraph of $(G-e)-v$, which is a contradiction since $G_2+b$ is nonplanar.
If $v=b$ then $(G-e)-v = G-v$, which implies $(G-e)-v$ is nonplanar since $G$ is NA, so this is a contradiction.
If $v \in \V(G_2)$ then $G_1+a$ is a subgraph of $(G-e)-v$, which is a contradiction since $G_1+a$ is nonplanar.
Therefore there is no apex for $G-e$ which is a contradiction.
So our assumption was wrong and one of $G_1+a$ and $G_2+b$ must be planar.
\end{proof}

We can now prove Theorem~\ref{thmgpagpb}.

\begin{proof} (of Theorem~\ref{thmgpagpb})
Let $G$ be as described.
By the first lemma we know that at least one of $G_1+a$ and $G_2+a$ must be nonplanar.
Without loss of generality suppose $G_2+a$ is nonplanar.
Since $G_2+a$ is nonplanar, we know that $G_1+b$ must be planar by the second lemma.
Since $G_1+b$ is planar, by the first lemma we know that $G_2+b$ is nonplanar.
By the second lemma this implies that $G_1+a$ must be planar.
Therefore, up to relabeling, $G_1+a$ and $G_1+b$ are both planar, and $G_2+a$ and $G_2+b$ are both nonplanar.
\end{proof}

Going forward, we adopt the convention suggested by Theorem~\ref{thmgpagpb} and label $G_1$ and $G_2$ such that $G_1 + a$, $G_1+b$ are planar and $G_2 +a$, $G_2 + b$ are not.
Let $G$ be MMNA with cut set $\{a, b \}$. Our next goal is to classify such graphs in case $ab$ is an edge of the graph.

\begin{theorem}
\label{thmGp1Gp2np}%
If $G$ is MMNA and $\kappa(G) = 2$ with cut set $\{a,b \}$ such that
$ab\in \E(G)$, then $G'_{1}$ and $G'_{2}$ are nonplanar.
\end{theorem}

\begin{proof}
Let $G'_{i}$ denote the induced subgraph on
$\V(G_i) \cup \{a,b\}$. By Theorem~\ref{thmgpagpb}, $G'_2$ is nonplanar.
For the sake of contradiction assume $G'_{1}$ is
planar. Since $G'_{2}$ is a proper subgraph of
$G$, there's a vertex $ v\in \V(G'_{2})$ such that $G'_{2}-v$ is planar.
But this means $G-v$ is planar and contradicts that $G$ is NA.

So if $G$ is MMNA with cut set $a,b\in \V(G)$ such
that $ab\in \E(G)$, then $G'_{1}$ and $G'_{2}$ are nonplanar.
\end{proof}

\begin{theorem}
\label{thmG1G2p}%
If $G$ is MMNA and $\kappa (G) = 2$ with cut set $\{a,b\}$ such that $ab \in \E(G)$,
then $G_1$ and $G_2$ are both planar.
\end{theorem}

\begin{proof}
By Theorem~\ref{thmGp1Gp2np}, $G'_1$ is nonplanar. By Theorem~\ref{thmG1G2np}, without loss of generality, $G_1$ is planar.
Suppose $G_2$ is nonplanar. Then $G'_1 \sqcup G_2$ is a proper subgraph of $G$. Since $G'_1$ and $G_2$ are both nonplanar, then $G'_1 \sqcup G_2$ has a disconnected MMNA minor, contradicting that $G$ is minor minimal.
\end{proof}

\begin{theorem}
\label{thmGp1mmnp}%
If $G$ is MMNA with cut set $\{a,b\}$
such that $ab\in \E(G)$, then $G'_{1} \in \{K_{5},K_{3,3}\}$.
\end{theorem}

\begin{proof}
First observe that for any $e \in  \E(G'_1)$, $ G'_1-e$ must be planar.
Suppose instead that there is $e' \in  \E(G'_1)$ such that
$G'_1-e'$ is nonplanar. Since $G-e'$ is apex, there's a vertex
$v \in \V(G)$ such that
$(G-e')-v$ is planar. However, $v\notin \{ a,b\} $
since $G_{2}+a$ and $G_2+b$ are nonplanar by Theorem~\ref{thmgpagpb}.
If $v\in \V(G_{1})$,
then $ G'_2$ is a subgraph of $(G-e')-v$.
By Theorem~\ref{thmGp1Gp2np}, since $G'_2$ is nonplanar, $(G-e')-v$ is also nonplanar.
If $v\in \V(G_{2})$, then
$ G'_1-e'$ is a subgraph of $(G-e')-v$, and since $ G'_1-e'$ is nonplanar,
$(G-e')-v$ is nonplanar. So we have a contradiction and deduce that
for all $ e\in E( G'_1)$, $ G'_1-e$ must be planar.

Since $G'_1$ is nonplanar by Theorem~\ref{thmGp1Gp2np},
and since $G'_1-e$ is planar for all $e \in G'_1$,
it follows that $G'_1$ consists of a K-subgraph
along with some number (possibly zero) of isolated vertices.
However if $ G'_1$ is anything other than $K_5$ or $K_{3,3}$,
then $G'_1$ has a proper minor $N \in \{K_5, K_{3,3} \}$ formed
by deleting isolated vertices or contracting edges in the K-subgraph.
Then $G$ has a proper minor $G'$ such that $N$ is a subgraph of $G'$.
Since $G$ is MMNA, there exists vertex $v \in \V(G')$ that is an apex.
Since $N$ and $G'_2$ are subgraphs of $G'$ and both $N$ and $G'_2$ are nonplanar,
we have that $v \in \V(N) \cap \V(G'_2) \subset \{ a,b\}$.
However $G'_2-a=G'_2+b$ and $G'_2-b=G'_2+a$ are both nonplanar
(Theorem~\ref{thmgpagpb}) and therefore $G$ has a proper NA minor.
This contradicts $G$ minor minimal.

So if $G$ is MMNA with cut set $\{a,b\}$ such that $ab \in \E(G)$,
then $G'_1 \in \{ K_{5},K_{3,3} \}$.
\end{proof}

\begin{theorem}
\label{thmcexists}%
If $G$ is MMNA with cut set  $\{a,b\}$ such
that $ab\in \E(G)$, then there is a vertex $ c\in \V(G_{2})$
such that every $a$-$b$-path in $G'_2-ab$ passes through $c$.
\end{theorem}

\begin{proof}
Assume for the sake of contradiction that there is no
such vertex $c$. Since $G$ is MMNA, $G-ab$ must have some apex $v$. If
$v\in \{ a,b \}$, then $(G-ab)-v=G-v$. This would mean that $G$ has an
apex, and contradicts that $G$ is NA. If $v\in \V(G_{1})$, then
$(G-ab)-v$ is nonplanar as it contains $G_{2}+a$,
which is nonplanar by Theorem~\ref{thmgpagpb}.
So it must be that $v\in \V(G_{2})$.
Then $ G'_1-ab$ is a subgraph of $(G-ab)-v$. Note that,
$G'_1-ab\in \{ K_{5}-e, K_{3,3}-e\} $ since
$ G'_1\in \{ K_{5},K_{3,3}\} $ by Theorem~\ref{thmGp1mmnp}.

Since there is no $c$ vertex as described in the statement of the theorem, there
remains an $a$-$b$-path in $(G'_2-ab)-v$. Together with $G'_1-ab$ this constitutes
a nonplanar subgraph of $(G-ab)-v$ contradicting the definition of $v$ as an apex for $G-ab$.
Thus, if $G$ is MMNA with $ab \in \E(G)$, then there is a vertex $c$ such that
every $a$-$b$-path of $G'_2-ab$ passes through $c$.
\end{proof}

\begin{theorem}
\label{thmabc}%
Let $G$ be MMNA with cut set $\{a,b\}$ and $ab \in \E(G)$ and let $c \in \V(G_2)$ be such that
every $a$-$b$-path of $G'_2-ab$ passes through $c$. Then $\{a,c\}$ and $\{b,c\}$
are also cut sets.
\end{theorem}

\begin{proof}
First we show there exists some $v_2 \in \V(G_2)$
such that $v_2 \neq c$, but $v_2$ is adjacent to $a$.
Suppose instead that $c$ is the only vertex in $G_2$ adjacent to $a$.
Since $G_2$ is planar by Theorem~\ref{thmG1G2p},
and since $G_2 + a$ has only one more edge than $G_2$, $G_2+a$ is also planar.
However this contradicts Theorem~\ref{thmgpagpb}
where $G_2 + a$ is shown to be nonplanar.

So let $v_2$ be a vertex of $G_2$ that is adjacent to $a$,
but is not $c$ and take $v_1 \in \V(G_1)$.
We demonstrate there is no $v_1$-$v_2$-path in $G -a,c$. Since any path from a vertex in $G_1$ to a vertex in $G_2$ must pass through $a$ or $b$ by assumption, the supposed path from $v_1$ to $v_2$ must pass through $b$, since $a$ has been deleted. However, there cannot be a path from $b$ to $v_2$ that does not pass through $c$. Otherwise we would be able to find a path from $b$ to $v_2$ and finally to $a$ without passing through $c$, violating our assumption on $c$. We conclude that $G - a,c$ is disconnected. By an analogous argument, $\{b,c\}$ is also a cut set for $G$.
\end{proof}

In order to classify connectivity two MMNA graphs with $ab \in \E(G)$, we need to describe $G'_1$ in case $ab \notin \E(G)$.

\begin{theorem}
\label{thmGp1mmnpme}%
If $G$ is MMNA with cut set $\{a,b\}$ such
that $ab\notin \E(G)$, then $ G'_1\in
\{ K_{5}-e, K_{3,3}-e, K_{3,3} \}$ and $G'_1+ab$ is nonplanar.
\end{theorem}

\begin{proof}
Let $G -a,b = G_1 \sqcup G_2$ and
let $G_{i}'$ denote the subgraph induced by
vertices $\V(G_{i})\cup \{ a,b\}$. If $G'_1$ is nonplanar,
then $G'_1$ has a K-subgraph $N$. Form a new
graph, $H$, by replacing $G'_1$ with $N$. It is clear that $a,b \in \V(N)$ because
if not, then $G$ contains two disjoint K-subgraphs ($G_2+a$  and $G_2+b$ are nonplanar, Theorem~\ref{thmgpagpb}) and therefore
has a proper MMNA minor.

We can see that $H$ is NA.
Take $v\in \V(H)$.
If $v\in \V(N-a,b)$, then $G_2+a$ is a subgraph of $H-v$ so $H-v$ is nonplanar.
If $v\in \V(G_2)$, then $N$ is a subgraph of $H-v$ so $H-v$ is nonplanar.
And if $v\in \{a,b\}$, then either $G_2+a$ or $G_2+b$ is a subgraph of $H-v$
and therefore $H-v$ is nonplanar.
Thus, $H$ is NA. Since $G$ is minor minimal, $G'_1=H$.
As $G$ is MMNA it has no degree two vertices and
since $ab \notin \E(G)$ $G'_1 = K_{3,3}$ in this case.

Suppose next that $G'_1$ is planar. Assume for the sake of contradiction
$G'_1+ab$ is
planar and replace $G'_1$ with the edge $ab$ to form a new
graph $H'$. Equivalently, $H'=G'_2+ab$. We observe that for every $ v\in \V(H')$, $H'-v$ is nonplanar.  If $v\in \{a,b\}$, then $G_2+a$ or $G_2+b$ is a subgraph of $H'-v$, which is then nonplanar. On the other hand if $v \in \V(G_2)$, then since $G$ is NA, $G-v$ has
a K-subgraph $M$. However
if $|\{a,b\}\cap \V(M)|<2$, then since $G'_1$ is planar, $M$ lies wholly in $G'_2$ and we may delete $G_1$ without changing $M$. That is, $M$ is a subgraph of
$H'-v$.
If $|\{a,b\}\cap V(M)|=2$, then by Lemma~\ref{lempath} $a$ and $b$ are vertices in a path of $M$. Since
$G'_1+ab$ is planar, we may replace $G'_1$ by $ab$ to create a new
K-subgraph, $B$ in  $H'-v$. Therefore $H'$ is NA. However as $H'$ is a proper minor of $G$, this is a contradiction. We conclude $G'_1+ab$ is nonplanar.

Finally, observe that  $G'_1+ab$ is a K-subgraph. Otherwise, we may
replace it with a K-subgraph contained in $G'_1+ab$ to get a proper minor
of $G$ that is NA. Since an MMNA graph
cannot have vertices of degree 2 or less, $G'_1+ab\in \{K_5,K_{3,3}\}$.

This shows if $G$ is MMNA with cut set $\{a,b\}$ such that
$ab\notin \E(G)$, then we have $G'_1\in\{K_{5}-e,K_{3,3}-e, K_{3,3} \}$.
\end{proof}

\begin{figure}[htb]
\begin{center}
\includegraphics[scale=0.3]{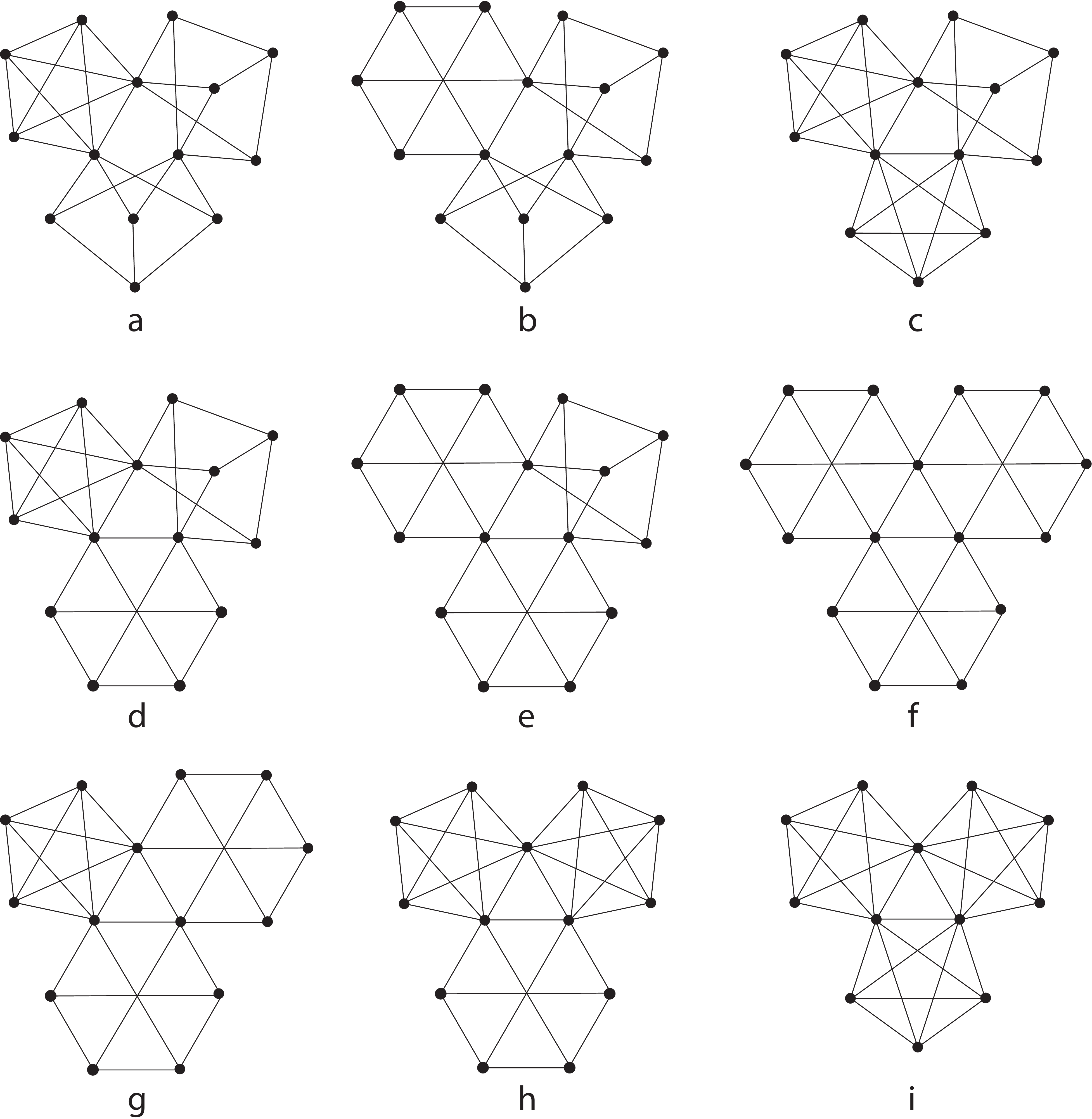}
\caption{The nine MMNA graphs with $ab \in \E(G)$.}\label{figabinE}
\end{center}
\end{figure}

\begin{theorem}
\label{thmabinE9}%
If $G$ is MMNA, $\kappa (G)=2$ with cut set $\{a,b\}$, and $ab\in \E(G)$, then $G$
is one of nine graphs shown in Figure~\ref{figabinE}.
\end{theorem}

\begin{proof}
It is straightforward to verify that the nine graphs are MMNA.
Let $G$ be MMNA, $\kappa (G)=2$ with cut set $\{a,b\}$, and $ab\in \E(G)$.
By Theorems~\ref{thmcexists} and \ref{thmabc},
there exists a vertex $c$ such that $\{a, c\}$
and $\{b, c\}$ are also $2$--cuts for $G$. Let $H_{1}'$ play the role of $G'_{1}$ for the $\{a,c\}$ cut set. That is, $G-a,c = H_1 \sqcup J_1$ with $H_1 +a$ and $H_1+c$ planar (see Theorem~\ref{thmgpagpb}). Similarly, let
$H'_{2}$ be the $G'_{1}$ for the $\{b,c\}$ cut set. By Theorem~\ref{thmGp1mmnp}, $G'_{1} \in \{K_{3, 3},K_{5}\}$ and
by that theorem and Theorem~\ref{thmGp1mmnpme},
$H'_{i}\in \{K_{3, 3},K_{3, 3}-e,K_{5}, K_{5}-e\}$.

Note that, if $H'_{1}$ is $K_{3, 3}-e$ or
$K_{5}-e$, then $G-b$ is planar and similarly for $H'_2$.
Thus, $H'_{1}, H'_{2} \in \{K_{3,
3}, K_{5}\}$. There are three cases depending on whether
$ac, bc\in \E(G)$ or not.

First suppose that $ab$ is the only one of $ab$, $bc$, and $ac$ present in the graph.
As above, $G'_{1}$, $H'_1$ and $H'_2$ are each either $K_{3, 3}$ or $K_{5}$.
However, by Theorem~\ref{thmGp1mmnpme}, this means $H'_1 = H'_2 = K_{3,3}$.
So, there are exactly two graphs (Graphs a and b in Figure~\ref{figabinE}) of this type, depending on whether $G'_1$ is
$K_5$ or $K_{3,3}$.

Next suppose that exactly one of $ac$ and $bc$, say $ac$, is in the graph.
As in the previous case $H'_{2} $ must be $K_{3, 3}$. There are three graphs
(Graphs c, d, and e of Figure~\ref{figabinE}) of this type as $\{G'_1, H'_1 \}$ is either $\{K_5, K_5 \}$, $\{K_5, K_{3,3} \}$, or $\{K_{3,3}, K_{3,3} \}$.

Finally, suppose all three edges $ab$, $ac$ and $bc$ are in the graph.
Then, as above, $G'_1$, $H'_{1}$, and
$H'_{2}$ are each either $K_{3, 3}$, or $K_{5}$.
There are four graphs of this type, shown as Graphs f through i of Figure~\ref{figabinE}. For example, such a graph has between zero and three $K_5$'s.

This shows that the nine graphs of Figure~\ref{figabinE} are the graphs where $G$ is MMNA, $\kappa(G)=2$ with cut set $\{a,b\}$, and $ab \in \E(G)$.
\end{proof}

\begin{figure}[htb]
\begin{center}
\includegraphics[scale=0.3]{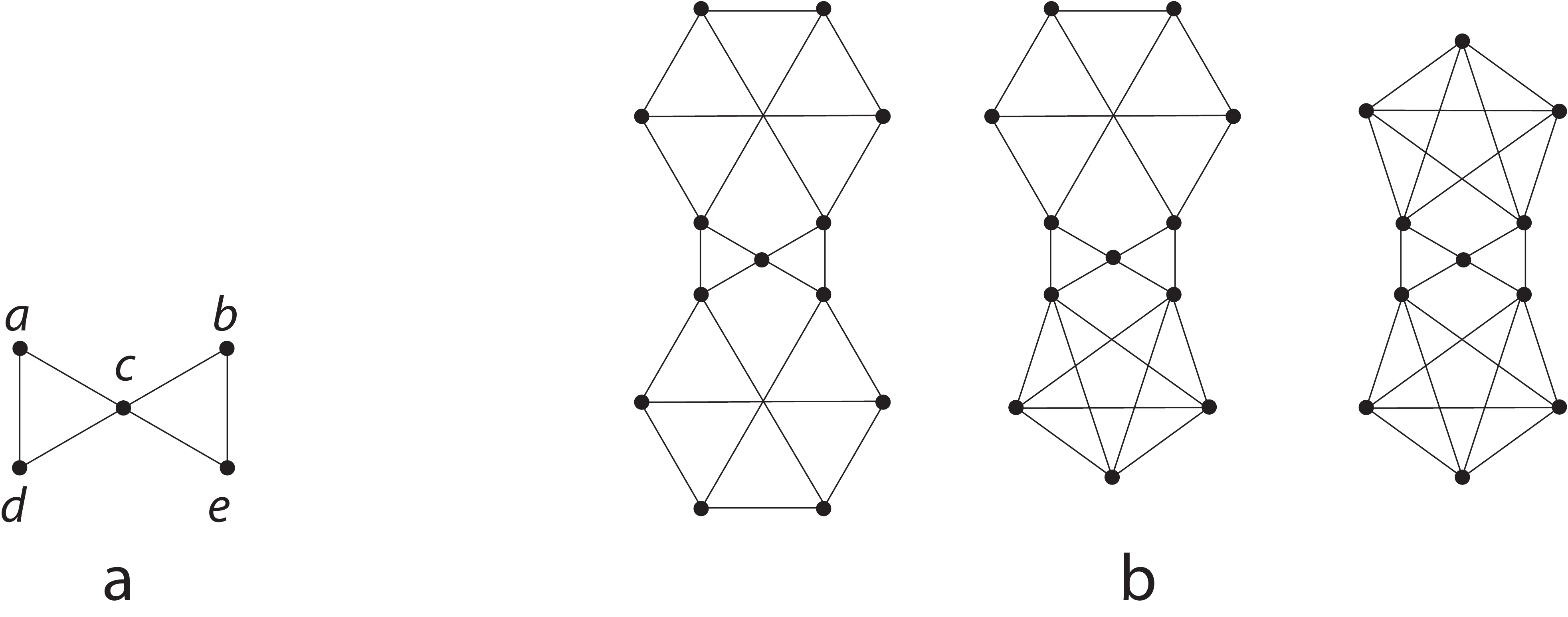}
\caption{Bowtie graphs.} \label{figbowtie}
\end{center}
\end{figure}

Henceforth, we can assume $ab \notin \E(G)$. By Theorem~\ref{thmGp1mmnpme}, this means $G'_1 \in \{K_5-e, K_{3,3}-e, K_{3,3}\}$.
We will say that $G$ is a \textbf{bowtie} if the neighborhood of $a,b$ in $G'_2$ is as shown in Figure~\ref{figbowtie}a. That is, $a$ and $b$ have degree two in $G'_2$ and $c$ has degree four. Although $d$ (respectively, $e$) has additional neighbors in $G'_2$ beside $\{a,c\}$
(resp., $\{b,c\}$), 
$de \notin \E(G'_2)$.

\begin{theorem}
\label{thmbowtie}%
If $G$ is a bowtie MMNA graph, then $G$ is one of the three graphs shown in
Figure~\ref{figbowtie}b.
\end{theorem}

\begin{proof}
It is straightforward to verify that the three graphs in the figure are MMNA.
Let $G$ be a bowtie MMNA graph. Then, referring to Figure~\ref{figbowtie}, $\{d,e\}$
is a cut set as well. Let $H'_{1}$ be `the $G'_{1}$' for the $\{d,e\}$ cut set. By Theorem~\ref{thmGp1mmnpme}, $G'_{1}$ and $H'_1$ are both drawn from $\{K_{3, 3},K_{3, 3}-e,  K_{5}-e\}$.

We will argue that neither is $K_{3,3}$.
For the sake of contradiction, assume instead
$G'_{1}=K_{3, 3}$. Notice $G'_{1}$ and $G_{2}$ are disjoint,
and nonplanar. So, $G$ has a proper NA minor,
$G'_{1} \sqcup G_{2}$, which contradicts that $G$ is to be minor minimal.

So, $G'_{1}$ and $H'_1$ are both in $ \{ K_{3, 3}-e, K_{5}-e\}$ where $ab$ is
the missing edge, $e$
and the only possibilities are the three graphs shown in Figure~\ref{figbowtie}b.
\end{proof}

Let $G$ be MMNA with cut set $\{a,b\}$ such that $ab \notin \E(G)$. We say $G$ is of {\bf $(2,2,c)$ type} if, in $G'_2$, $a$ and $b$ are of degree two and have
$c$ common neighbors. For example, a bowtie graph is $(2,2,1)$ type.

\begin{figure}[htb]
\begin{center}
\includegraphics[scale=0.3]{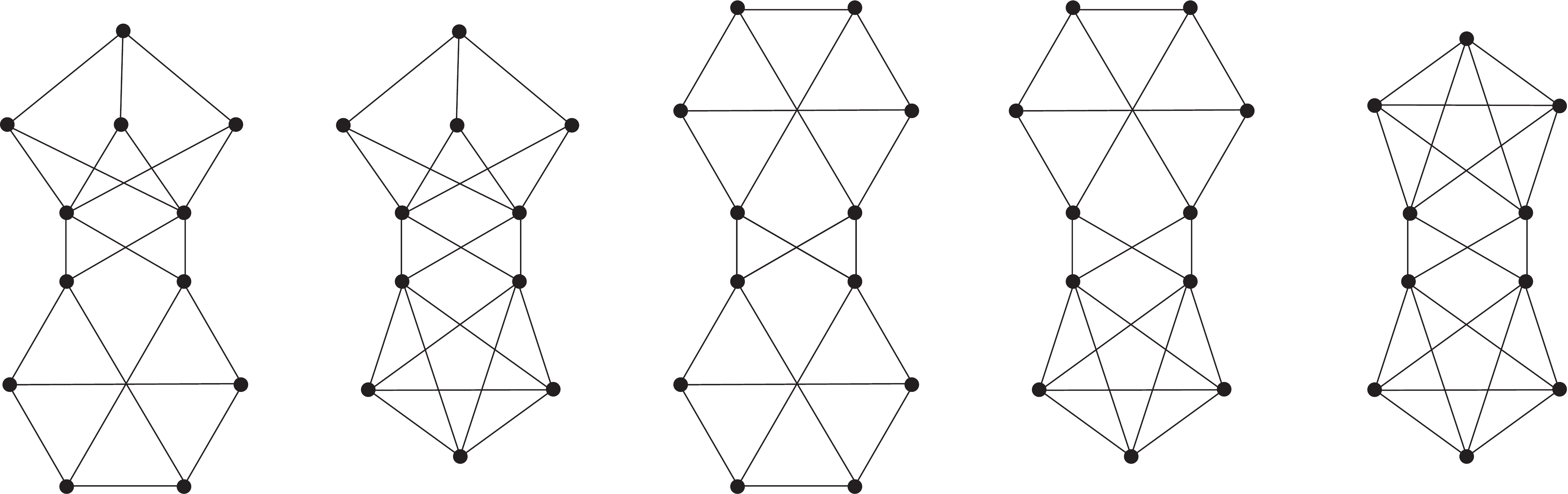}
\caption{Graphs of type $(2,2,2)$.} \label{fig222}
\end{center}
\end{figure}

\begin{theorem}
\label{thm222}%
If $G$ is MMNA and of $(2, 2, 2)$
type, then $G$ is one of the five graphs shown in Figure~\ref{fig222}.
\end{theorem}

\begin{proof}
It is straightforward to verify that the five graphs are MMNA.
Let $G$ be MMNA with cut set $\{a,b\}$ and of
$(2, 2, 2)$ type. Let $\{c,d\}$ be the common neighbors of $a$ and $b$ in $G'_2$.
Note that $cd \notin \E(G)$ as otherwise, $G$ must be one of the nine graphs
of Theorem~\ref{thmabinE9} and none of those are $(2,2,2)$ type.

By Theorem~\ref{thmGp1mmnpme}, and using symmetry, $G'_{1},
G_{2}\in \{ K_{3, 3},K_{3, 3}-e, K_{5}-e\}$. However, they cannot both
be $K_{3,3}$ as otherwise $G'_{1} \sqcup G_{2}$ is
a proper NA subgraph, which contradicts that $G$ is minor minimal. So
at most one of the subgraphs can be $K_{3, 3}$.
This leaves the five possibilities shown in Figure~\ref{fig222}.
\end{proof}

\begin{theorem}
\label{thmnoc}%
Suppose $G$ is MMNA and of connectivity two with
$G'_1\in \{ K_{5}-e,K_{3,3}-e\}$.
Then there is no vertex, other than $a$ and $b$, common to all
$a$-$b$-paths in $G'_2$.
\end{theorem}

\begin{proof}
Assume, for the sake of contradiction, that
$G'_1\in \{ K_{5}-e,K_{3,3}-e\}$ and there's a vertex $ c\in \V(G_{2})$
that lies on every $a$-$b$-path in  $G'_2$. Then, as in Theorem~\ref{thmabc},
$\{a,c\}$ and $\{b,c\}$ are $2$--cuts for $G$, and as in the proof of
Theorem~\ref{thmabinE9} we can define $H'_1$ as ``the $G'_1$ for the $\{a,c\}$ cut''
and similarly $H'_2$ for the $\{b,c\}$ cut and, by Theorems~\ref{thmGp1mmnp} and
\ref{thmGp1mmnpme} both $H'_1$ and $H'_2$ are drawn from $\{K_5, K_{3,3}, K_5-e, K_{3,3}-e \}$. Then $G-c$ is planar, contradicting that $G$ is NA.

Therefore, if $G$ is MMNA, of connectivity 2 with
$G'_1\in \{ K_{5}-e,K_{3,3}-e\}$,
then there is no vertex, other than $a$ and $b$, common to all
$a$-$b$-paths in $G'_2$.
\end{proof}

\begin{theorem} 
\label{thmindabP}
Let $G$ be MMNA with $\K(G) = 2$ and $ab \notin E(G)$ where 
$\{a,b\}$ is a $2$-cut. If $G_2$ is nonplanar, then there are independent $a$-$b$-paths in $G'_2$.
\end{theorem} 

\begin{proof} By Theorem~\ref{thmGp1mmnpme}, $G'_1 \in \{K_5-e, K_{3,3}, K_{3,3}-e \}$. 
However, if $G'_1 = K_{3,3}$ then, together with $G_2$, this constitutes a pair of disjoint $K$-subgraphs, 
which would mean $G$ has a proper disconnected NA minor, a contradiction. So $G'_1 \in \{K_5-e, K_{3,3}-e \}$ and we can apply Menger's theorem and Theorem~\ref{thmnoc}.
\end{proof}

\begin{figure}[htb]
\begin{center}
\includegraphics[scale=0.3]{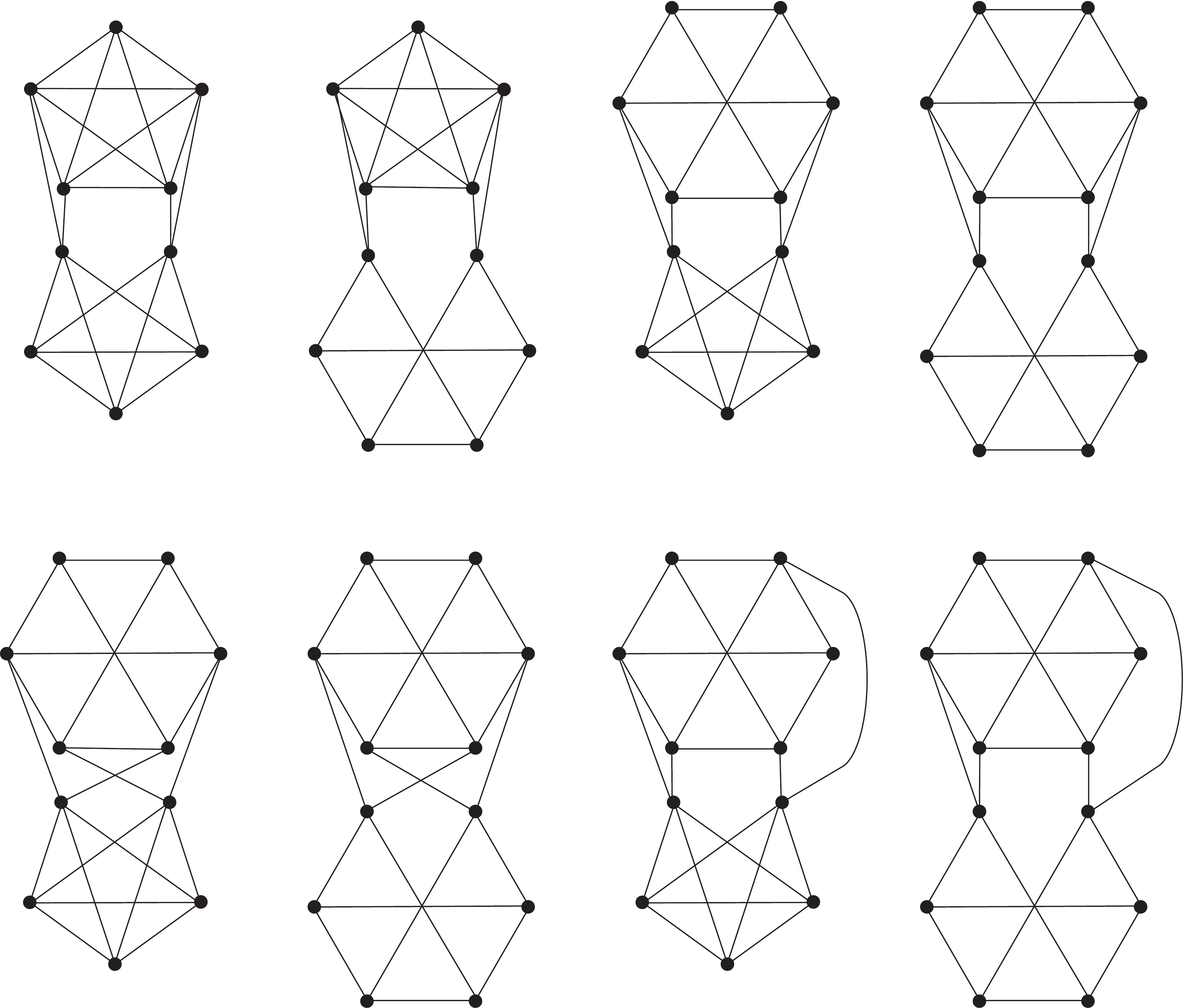}
\caption{Graphs of type $(2,2,0)$.} \label{fig220}
\end{center}
\end{figure}

\begin{theorem}
\label{thm220}%
If $G$ is MMNA of $(2,2,0)$ type and $G_2 \in \{K_5, K_{3,3} \}$, then $G$ is one of the eight graphs in Figure~\ref{fig220}.
\end{theorem}

\begin{proof}
Notice that the eight graphs in the figure are MMNA.
Suppose $G$ is MMNA of $(2,2,0)$ type with $G_2$ a Kuratowski graph.
By Theorem~\ref{thmGp1mmnpme},
$G'_1 \in \{K_5-e, K_{3,3}, K_{3,3}-e\}$. However,
$G'_1$ cannot be $K_{3,3}$ because then, together with $G_2$
it forms a disconnected MMNA minor of $G$. We continue by examining the ways
to construct $G'_2$.

To construct $G'_2$ we consider how to add the vertices $a$ and
$b$ to $G_2$. Let $a$ have neighbors
$v_1,v_2 \in \V(G_2)$ and let $v_3,v_4 \in \V(G_2)$
be the neighbors of $b$. Since $G$ is $(2,2,0)$
$\{ v_1, v_2 \} \cup \{v_3, v_4 \} = \emptyset$.
Up to symmetry, there is only one way to attach
$a$ and $b$ to $K_5$. This gives two of the graphs in the figure as $G'_1$ is either $K_5-e$ or $K_{3,3}-e$.

In $K_{3,3}$ the vertices are split into two parts $A$ and $B$, each of three
vertices. Then the four vertices $v_i$,
$i = 1, \ldots 4$ are either divided with two in each part, or else with three in one part and the fourth in the other. In the first case, there are two subcases: either
$\{v_1, v_2 \} \subset A$ (and $\{v_3,v_4 \} \subset B$) or else
$|\{v_1 ,v_2\} \cap A | = |\{v_1,v_2 \} \cap B| = 1$
(and similarly for $\{v_3, v_4 \}$). These three choices for $G'_2$ along with
the two choices for $G'_1$, either  $K_5 -e$ or $K_{3,3}-e$, account
for the remaining six graphs in Figure~\ref{fig220}.
\end{proof}

\begin{figure}[htb]
\begin{center}
\includegraphics[scale=0.3]{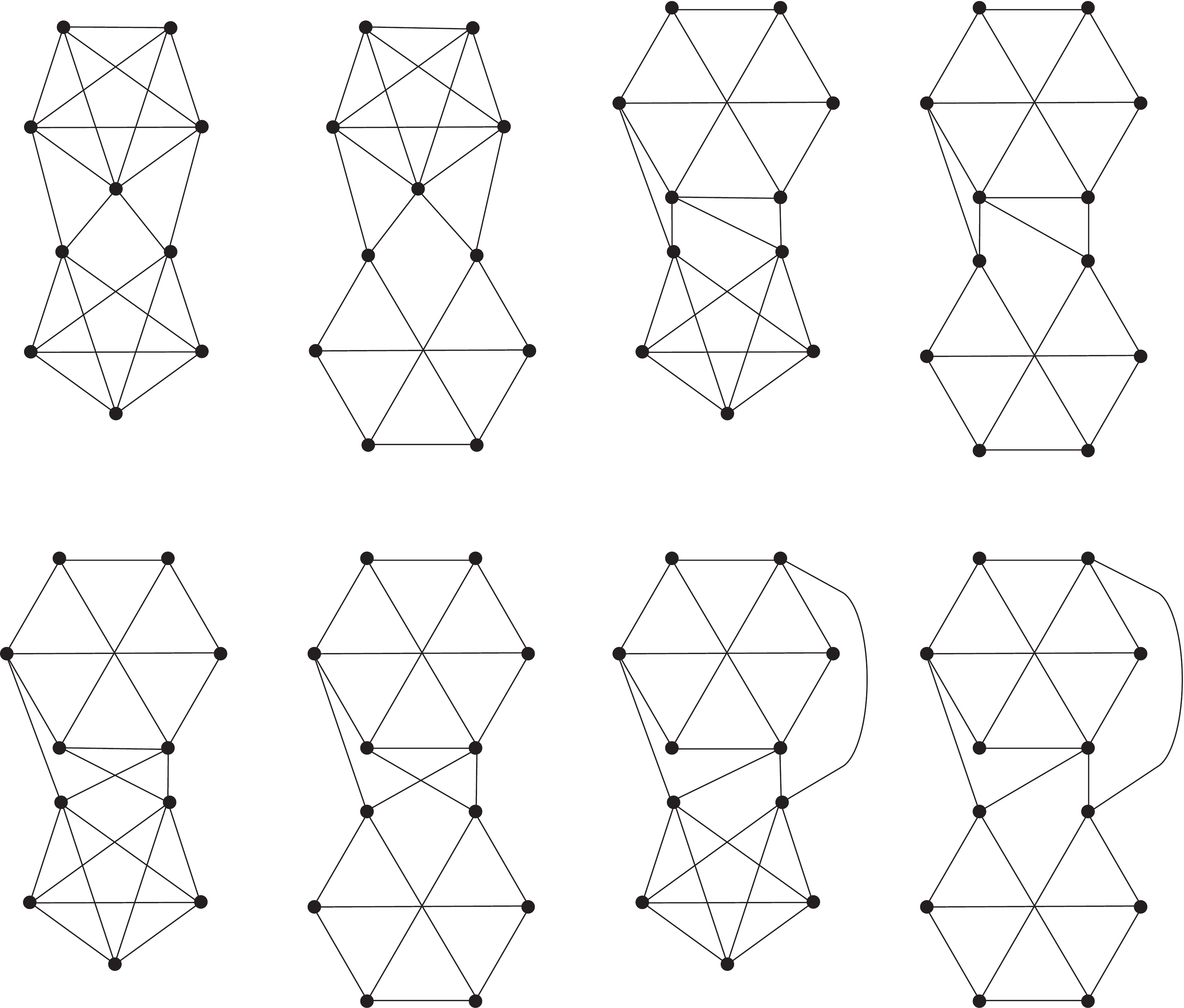}
\caption{Graphs of type $(2,2,1)$.} \label{fig221}
\end{center}
\end{figure}

\begin{theorem}
\label{thm221}%
If $G$ is MMNA of $(2,2,1)$ type and $G_2 \in  \{K_5,K_{3,3} \}$, then $G$ is
one of the eight graphs of Figure~\ref{fig221}.
\end{theorem}

\begin{proof}
The proof is similar to that for $(2,2,0)$ type. If $G_2$ is a Kuratowski graph,
then $G'_1$ cannot be $K_{3,3}$ as that would result in a proper NA minor.
So $G'_1 \in \{K_5-e, K_{3,3} -e \}$. If $G_2 = K_5$, up to symmetry there's only
one way to form $G'_2$ and this gives two graphs in the figure as $G'_1$ is either
$K_5-e$ or $K_{3,3}-e$.

If $G_2 = K_{3,3}$, there are three ways to form $G'_2$. Together, $a$ and $b$ have three neighbors in $G_2$ which can either all lie in one part or else be split
with a single vertex in one part and the remaining two in the other. In this second
case, there are two further subcases since the vertex
that is alone in its part can either
be the common neighbor or not. Together with these three choices for $G'_2$,
there are two choices for $G'_1$, either $K_5-e$ or $K_{3,3}-e$. This gives
the remaining six graphs of Figure~\ref{fig221}
\end{proof}

We conclude this section with a classification of the MMNA graphs of connectivity two, with 2--cut $\{a,b\}$, such that $G-a,b$ has a nonplanar component. By Theorem~\ref{thmG1G2p} we must have $ab \notin \E(G)$, and by 
Theorem~\ref{thmgpagpb}, $G_1$ is planar. In other words, 
if there is a nonplanar component, it must be $G_2$. So far, we have constructed 21 graphs with nonplanar $G_2$, the three 
bowtie graphs of Theorem~\ref{thmbowtie}, two of the $(2,2,2)$ graphs (the two to the left of Figure~\ref{fig222}), and eight each of $(2,2,0)$ type (Theorem~\ref{thm220}) and $(2,2,1)$ type (Theorem~\ref{thm221}). This is in fact a complete 
listing of the graphs with $G_2$ nonplanar as we now show. 

\begin{theorem}
Let $G$ be MMNA with $\K(G) = 2$ and 2-cut $\{a,b\}$ such that $G-a,b$ has a 
nonplanar component. Then $G$ is of $(2,2,c)$ type with $c = 0$, $1$, or $2$ and
appears in one of Figures~\ref{figbowtie}b, \ref{fig222}, \ref{fig220}, or \ref{fig221}.
\end{theorem} 

\begin{proof}
Assume the hypothesis. As remarked above, if $\{a,b\}$ is a 2--cut, this implies
$ab \notin \E(G)$ and $G_2$ is nonplanar.
Let $H_2$ be a K-subgraph of  $G_2$.
Since $ab \notin E(G)$, combining Theorem~\ref{thmGp1mmnpme} and 
Theorem~\ref{thmMMNAdisc}, we
have $G'_1 \in \{K_5-e, K_{3,3}-e\}$.
By Theorem~\ref{thmindabP} there are independent $a$-$b$-paths in $G'_2$, call them $P_1$ and $P_2$. Since, by Theorem~\ref{thmGp1mmnpme}, $G'_1 + ab$ is nonplanar, $P_1$ and $P_2$ each have vertices in common with $H_2$. 
(Otherwise, $G$ has disjoint nonplanar subgraphs and therefore a disconnected NA 
minor, by Theorem~\ref{thmMMNAdisc}, contradicting $G$ being minor minimal.)
By contracting edges if necessary, we have a minor of $G$ for which 
the vertices of $P_i$
are $a, a_i, \ldots b_i, b$ with $a_i,b_i \in V(H_2)$, $i = 1,2$.
Then there are several cases that correspond to $(2,2,c)$ type where $c = 0,1,2$. 

Suppose first that $a_1 = b_1$ and $a_2 = b_2$ so that $G$ is of $(2,2,2)$ type. By contracting edges in $H_2$ if needed, we recognize that $G$ has one of the five
graphs of Theorem~\ref{thm222} as a minor. 
Since $G$ is MMNA, $G$ is one of these
five graphs and since $G_2$ is nonplanar, $G$ must be one of the two graphs with
$G_2 = K_{3,3}$ (i.e., the two to the left of Figure~\ref{fig222}). In other words $G$ is of $(2,2,2)$ type and appears in one of the figures, as required.

The rest of the argument is a little technical and we introduce some 
notation to simplify the exposition. The K-subgraph $H_2$ is a subdivision of 
$K_5$ or $K_{3,3}$ and, along with vertices of degree two, 
has five or six vertices of higher degree 
that we will call \textbf{branch vertices}. Corresponding to the edges of 
$K_5$ or $K_{3,3}$, the branch vertices are connected by paths
 that we call 2-\textbf{paths} whose internal vertices are all of degree two.

To continue the argument, 
suppose next that, say, $a_1 = b_1$, but $a_2 \neq b_2$. By contracting edges in
$H_2$ if necessary, we can arrange that at least two of the three vertices $a_1$, $a_2$, and $b_2$ become branch vertices.
of the K-subgraph. If all three can be made branch vertices, then, by further edge contractions, if necessary, we see that one of the eight $(2,2,1)$ graphs 
of Theorem~\ref{thm221} is a minor of $G$. Since $G$ is MMNA, this means $G$
is one of the $(2,2,1)$ graphs with $G_2 \in \{K_5, K_{3,3}\}$ appearing in 
Figure~\ref{fig221}, as required. If not, we can assume that it is 
$a_1$ that remains as a degree two vertex of $H_2$. For, if it
is $a_2$ or $b_2$ that remains, we can contract edges to make $a_2 = b_2$ and
return to the previous case. With $a_1$ as a degree two vertex in $G_2$, 
we recognize that, perhaps by further edge contractions, $G$ has a bowtie graph as 
a minor. Since $G$ is MMNA, $G$ is a bowtie graph. That is $G$ is of $(2,2,1)$ type
and appears in Figure~\ref{figbowtie}b, as required.

Finally, suppose $a_1 \neq b_1$ and $a_2 \neq b_2$. If all four can be made 
distinct branch vertices by edge contractions in $H_2$, then $G$ has a 
$(2,2,0)$ minor, so $G$ is a $(2,2,0)$ graph with 
$G_2 \in \{K_5, K_{3,3}\}$ appearing in 
Figure~\ref{fig220}, as required. 

Next, suppose at most three can be made into branch vertices and, 
without loss of generality, suppose it's $a_1$ that remains as a degree 
two vertex in $H_2$. This means
$a_1$ lies on a $2$-path between two of $b_1$, $a_2$,
and $b_2$. If the path ends at $b_1$, by further edge contractions in $H_2$, we
can realize $a_1 = b_1$ as a branch vertex and return to an earlier case. 
So, we can assume that $a_1$ is on a $2$-path between $a_2$ and $b_2$. Use the
part of the $2$-path between $a_1$ and $b_2$ to form a new $a$-$b$-path $P'_1$
(i.e., $a'_1 = a_1$ and $b'_1 = b_2$) 
and use a path in $H_2$ 
between the branch vertices $a_2$ and $b_1$ 
that avoids the branch vertex $b_2$ to construct an independent 
$a$-$b$-path $P'_2$ (i.e., $P'_2$ has $a'_2 = a_2$ and $b'_2 = b_1$).
Now we can contract edges in $P'_1$ to identify $a'_1 = a_1$ and 
$b'_1 = b_2$ to return to the 
earlier case where $a_1 = b_1$. This completes the argument when at most three of the vertices can be moved to branch vertices.

Finally, suppose that at most two of the vertices can be made into branch vertices of
$H_2$ by contracting edges, if needed. There are two subcases. If $a_1$ and $b_1$ are the branch vertices, then $a_2$ and $b_2$ are degree two vertices
on a $2$-path between $a_1$ and $b_1$. Here
we can further contract edges in $H_2$ to identify $a_2$ and $b_2$, 
which returns us to an earlier case. In the second subcase, 
without loss of generality, it is $a_1$ and $a_2$ that are the branch
vertices of $H_2$. Assuming we can not easily contract edges to identify $a_1$ and $b_1$ or $a_2$ and $b_2$, it must be that the $2$-path from $a_1$ to $a_2$ passes first through $b_2$ and then through $b_1$. In this case, we replace $P_1$ and $P_2$ by the independent paths $P'_1$ which uses the $2$-path from $a_1$ to $b_2$ (so $a'_1 = a_1$ and $b'_1 = b_2$), and $P'_2$ which uses the $2$-path from $a_2$ to $b_1$ (then $a'_2= a_2$ and $b'_2 = b_1$). By further edge 
contractions, we return to our first case where $a_1 = b_1$ and $a_2 = b_2$. 
\end{proof}

Together, the three bowtie graphs and the eight of Figure~\ref{fig221} give
eleven MMNA graphs of $(2,2,1)$ type.
In total we have found three disconnected MMNA graphs,
nine where $ab \in \E(G)$,
as well as eight, eleven, and five, respectively when $G$ is of type $(2,2,c)$ for
$c = 0$, $1$, $2$, respectively. This gives a total of 36 MMNA graphs.


\section{MMNE and MMNC Graphs}
In this section we classify MMNE and MMNC graphs of
connectivity, $\K(G)$, at most one.
For MMNE graphs we also show $\K(G) \leq 5$ and determine
the graphs with $\K(G) = 2$ and minimum degree at least three.
We conclude the section by  describing a computer
search that found 55 MMNE and 82 MMNC graphs.

We begin by observing that the MMNE and MMNC graphs are not Kuratowski sets as the opposite properties are not minor closed. Recall that NE is an abbreviation for not edge apex. The opposite property is \textbf{edge apex} meaning there's an
$e \in \E(G)$ so that $G-e$ is planar.  We call such an edge an \textbf{apex (edge)}.
Similarly, the opposite of NC is \textbf{contraction apex} meaning there's an
edge $e$ such that $G/e$ is planar. We call $e$ a \textbf{contraction apex}.

\begin{theorem}
\label{thmmmnenotclosed}%
    Deleting an edge of an edge apex graph results in an edge apex graph.
    Contracting an edge of an edge apex graph results in an edge apex graph
    unless the edge that is contracted is the \emph{only} apex edge.
\end{theorem}
\begin{proof}
    Suppose that $G$ is edge apex,
    so it contains an edge $e$ such that $G-e$ is planar.
    Let $G'$ be the results of deleting some edge $f$ in $G$.
    If $f \neq e$, consider $G'-e$
    and note that $G'-e = G-e,f$ which is a minor of $G-e$.
    Graph $G-e$ is planar, so $G'-e$ is also planar,
    and $e$ is an apex for $G'$, which is therefore edge apex.
    Otherwise, if $f = e$,
    then $G'$ would be planar and so would also be edge apex.

    Now suppose that $G$ contains at least two edges $e_1$ and $e_2$
    ($e_1 \neq e_2$) such that both $G-e_1$ and $G-e_2$ are planar.
    Let $f$ be an arbitrary edge in $G$
    and let $G''$ be the result of contracting edge $f$ in $G$.
    Without loss of generality, suppose that $f \neq e_1$.
    Consider the graph $G''-e_1$, where if $e_1$ is incident to $f$ in $G$
    then $e_1$ is incident to the vertex formed by contracting $f$ in $G''$.
    Note that this graph $G''-e_1$ is a minor of $G-e_1$.
    But $G-e_1$ is planar, and since planarity is closed under taking minors,
    the graph $G''-e_1$ is planar.
    So edge $e_1$ is an edge-apex of $G''$.
\end{proof}

\begin{figure}[htb]
\begin{center}
\includegraphics[scale=0.3]{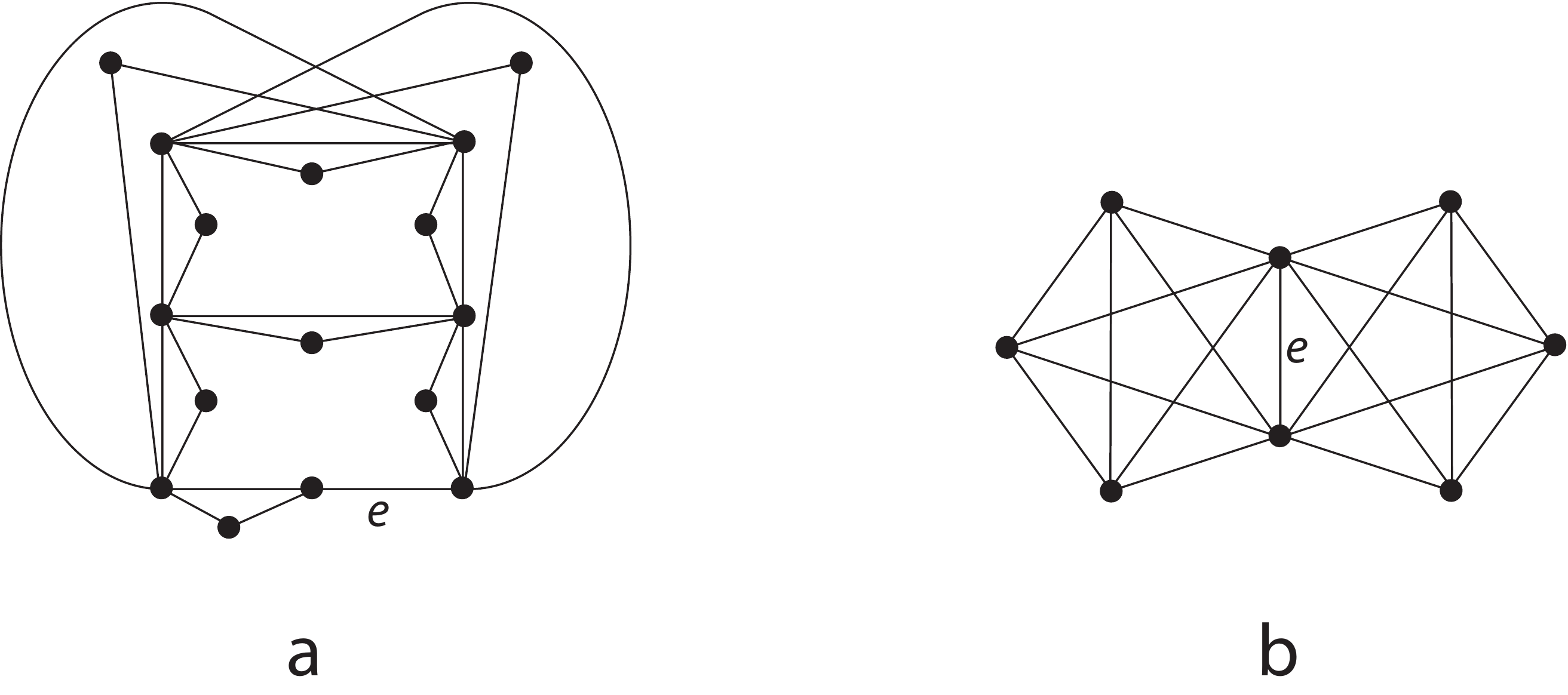}
\caption{Examples showing MMNE and MMNC are not 
Kuratowski sets.
}\label{figNECnc}
\end{center}
\end{figure}

\begin{theorem}
\label{thmmmnencex}%
The set of graphs that are edge apex is not closed under taking minors.
\end{theorem}

\begin{proof}
Let $G$ be the graph of Figure~\ref{figNECnc}a, essentially a $K_{3.3}$ with all but one edge replaced 
 a triangle, 
the final edge being subdivided into a further triangle
and edge $e$.
This graph is edge apex with $e$ as the unique apex. However, $G/e$ is $K_{3,3}$ with every edge replaced by a triangle. So, $G/e$ is not edge apex.
\end{proof}

\begin{theorem}
    Contracting an edge of a contraction apex graph
    results in a contraction apex graph.
    Deleting an edge of a contraction apex graph
    results in a contraction apex graph
    unless the edge that is deleted is the \emph{only} contraction apex
\label{thmmmncnotclosed}%
\end{theorem}
\begin{proof}
    Suppose that $G$ is contraction apex,
    so it contains an edge $e$ such that $G/e$ is planar.
    Let $G'$ be the result of contracting some edge $f$ in $G$.
    If $f \neq e$, consider $G'/e$
    and note that $G'/e = G/e,f$ which is a minor of $G/e$.
    Graph $G/e$ is planar, so $G'/e$ is also planar,
    and $e$ is a contraction apex for $G'$, which is therefore a contraction apex graph.
    Otherwise, if $f = e$,
    then $G'$ would be planar and so would also be contraction apex.

    Now suppose that $G$ contains at least two edges $e_1$ and $e_2$
    ($e_1 \neq e_2$) such that both $G/e_1$ and $G/e_2$ are planar.
    Let $f$ be an arbitrary edge in $G$
    and let $G''$ be the result of deleting edge $f$ in $G$.
    Without loss of generality, suppose that $f \neq e_1$.
    Consider the graph $G''/e_1$ and note that it is a minor of $G/e_1$.
    But $G/e_1$ is planar, and since planarity is closed under taking minors,
    the graph $G''/e_1$ is planar.
    So edge $e_1$ is a contraction apex of $G''$.
\end{proof}

\begin{theorem}
\label{thmmmncncex}%
The set of graphs that are contraction apex is not closed under taking minors.
\end{theorem}

\begin{proof}
Define the graph $G$ as two copies of $K_5$ that share a common edge $e$
(see Figure~\ref{figNECnc}b).
We show that $G$ is contraction apex, but has a minor that is NC.
Indeed, contracting the common edge,
$G/e = K_4 \dotcup K_4$, which is planar.  Note that
this is the unique contraction apex of $G$.

Now define the subgraph $G'$ as $G-e$. Label the two subgraphs isomorphic to $K_5 - e$ as $G_1$ and $G_2$. Without loss of generality, suppose we contract an edge $f$ in $G_2$. Notice that we are left with $G_1 = K_5-e$, and a path through $G_2$ that connects the two degree three vertices of $G_1$. Thus, $G'/f$ has a subgraph homeomorphic to $K_5$ and is nonplanar. By symmetry, whatever edge $f \in \E(G')$ we choose, $G'/f$ is nonplanar. Thus $G'$ is NC.
\end{proof}

We next classify the disconnected and connectivity one MMNE and MMNC graphs, which turn out to be the same sets.

\begin{theorem}
The disconnected MMNE graphs are $K_5 \sqcup K_5$, $K_5 \sqcup K_{3,3}$, and $K_{3,3} \sqcup K_{3,3}$.
\end{theorem}

\begin{proof}
First observe that these three graphs are MMNE.
Let $G$ be MMNE and disconnected. Suppose one of $G_1, G_2$ is planar, say
$G_1$. Then let $e_1 \in \E(G_1)$, and note that $G-e_1$ is not NE and nonplanar.
Let $e_2$ be the edge whose removal from $G-e_1$ gives a planar graph.
Since $G_1$ is planar, it must be that $e_2$ is  in $\E(G_2)$.
But, since $G_1$ is planar, this means that removing $e_2$ from $G$ gives the disconnected union of the planar $G_1$ and a planar minor of $G_2$.
So, this graph, $G-e_2$, is planar, which is a contradiction since $G$ is NE.
So it must be that $G_1$ and $G_2$ are both nonplanar.
Thus one of the graphs generated by $G_1 \sqcup G_2$ where $G_1,G_2 \in \{ K_5, K_{3,3} \}$ must be a minor of $G$.
Since $G$ is minor minimal, $G$ must be one of these three graphs.
\end{proof}

\begin{theorem}
The disconnected MMNC graphs are $K_5 \sqcup K_5$, $K_5 \sqcup K_{3,3}$, and $K_{3,3} \sqcup K_{3,3}$.

\end{theorem}

\begin{proof}
First observe that these three graphs are MMNC.
Let $G$ be MMNC and disconnected. Suppose one of $G_1, G_2$ is planar, say $G_1$.
Then let $e_1 \in \E(G_1)$, and note that $G-e_1$ is not NC and nonplanar.
Then there is an edge
$e_2 \in \E(G-e_1)$ such that $(G-e_1)/e_2$ is planar.
Since $G_1$ is planar, it must be that $e_2$ is  in $\E(G_2)$.
But, since $G_1$ is planar, this means that contracting $e_2$ in $G$ gives the disconnected union of the planar $G_1$ and a planar minor of $G_2$.
This graph $G/e_2$ is planar, which is a contradiction since $G$ is NC.
So it must be that $G_1$ and $G_2$ are both nonplanar. Then one of the graphs
$G= G_1 \sqcup G_2$  with $G_i \in \{ K_5, K_{3,3} \}$ is a minor of $G$.
Since $G$ is minor minimal, it is one of
those three graphs.
\end{proof}

\begin{cor}
Let $G$ be disconnected. The following are equivalent: $G$ is MMNA; $G$ is MMNE; $G$ is MMNC.
\end{cor}

Recall that $G_1 \dotcup G_2$ is the union of $G_1$ and $G_2$ with one vertex identified.

\begin{theorem}
	If $G$ is MMNE and $\K(G)=1$ then $G = G_1 \dotcup G_2$ where $G_1, G_2 \in \{ K_5, K_{3,3} \}$, and they share exactly one vertex.
\end{theorem}

\begin{proof}
First observe that these three graphs are MMNE.
Let $G= G_1 \dotcup G_2$ and suppose for the sake of contradiction that one of $G_1$ and $G_2$, say $G_1$, is planar.
Let $e$ be an edge of $G_1$. Then $G-e$ is not NE and nonplanar.
Let $f$ be the edge apex of $G-e$. Since $G_1$ is planar, $f$ must lie in $\E(G_2)$.
Since $G_2-f$ is a subgraph of the planar $G-e,f$, it must itself be planar.
Note that $G-f = G_1 \cup (G_2-f)$ is the union of two planar graphs that share at most one vertex, which is clearly planar.
This is a contradiction, since $G$ is NE. So both $G_1$ and $G_2$ are nonplanar.
So $G$ has one of the graphs $G_1 \dot{\cup} G_2,\,\, G_1, G_2 \in \{ K_5, K_{3,3} \}$ as a minor.
Since these graphs are NE and $G$ is minor minimal, $G$ must be one of these three graphs.
\end{proof}

\begin{theorem}
	If $G$ is MMNC and $\K(G)=1$ then $G = G_1 \dotcup G_2$ where $G_1, G_2 \in \{ K_5, K_{3,3} \}$, and they share exactly one vertex.
\end{theorem}

\begin{proof}
First observe that these three graphs are MMNC.
Let $G = G_1 \dotcup G_2$ and suppose for the sake of contradiction that one of $G_1$ and $G_2$, say $G_1$, is planar.
Let $e$ be an edge of $G_1$. Then $G-e$ is not NC and nonplanar.
Let $f \in \E(G-e)$ be the contraction apex of $G-e$, that is, $(G-e)/f$ is planar. Since $G_1$ is planar, $f$ must lie in $G_2$.
Since $G_2/f$ is a subgraph of the planar $(G-e)/f$, it must itself be planar.
Note that $G/f = G_1 \cup (G_2/f)$  is the union of two planar graphs that share at most one vertex, which is clearly planar.
This is a contradiction, since $G$ is NC.

Thus, both $G_1$ and $G_2$ are nonplanar.
So $G$ has one of the graphs $G_1 \dot{\cup} G_2$ with $G_1, G_2 \in \{ K_5, K_{3,3} \}$ as a minor.
Since these graphs are NC and $G$ is minor minimal, $G$ must be one of these three graphs.
\end{proof}

\begin{cor} Let $G$ have connectivity one. Then $G$ is MMNE if and only if it is MMNC.
\end{cor}

Recall that there are no MMNA graphs of connectivity one. In particular, for each of
$K_5 \dotcup K_5$, $K_5 \dotcup K_{3,3}$, and $K_{3,3} \dotcup K_{3,3}$, the cut vertex is an apex.
We next classify the MMNE graphs of connectivity two under the assumption
that the minimum degree, $\delta(G)$, is at least three.
We will argue that there are exactly six such graphs and we begin with the observation that those graphs are indeed MMNE.
As discussed at the end of this section, based on a computer search,
these again coincide with the MMNC examples of connectivity two with $\delta(G) \geq 3$. In addition to being both MMNE and MMNC, these 12 graphs with $\K(G) \leq 2$ are exactly the 
obstructions,
of connectivity at most two, to embedding a graph in the projective plane, see \cite[Section 6.5]{MT}.

\begin{figure}[htb]
\begin{center}
\includegraphics[scale=0.3]{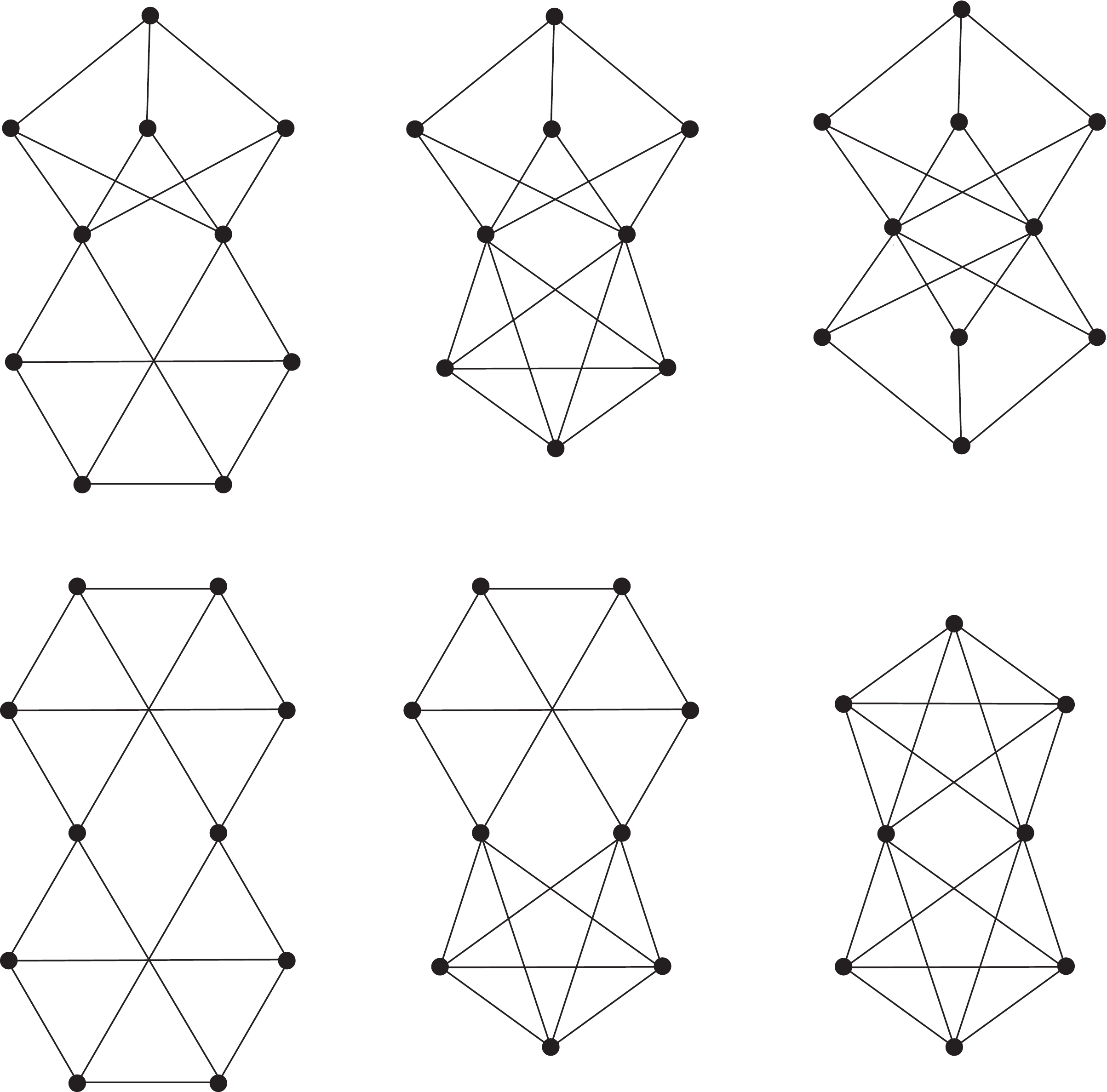}
\caption{The six MMNE graphs of connectivity two with $\delta(G) \geq 3$.}\label{figMMNEK2}
\end{center}
\end{figure}

\begin{theorem}
\label{thmMMNEK21}%
The six graphs of Figure~\ref{figMMNEK2} are MMNE.
\end{theorem}

Note that these graphs are of the form $G_1 \ddotcup G_2$ with $G_i \in
\{K_5-e, K_{3,3}, K_{3,3}-e \}$.

\begin{proof}
Let $G$ be one of the six graphs and $e$ denote an arbitrary edge of $G$.
It's easy to verify that each $G-e$ is nonplanar, so $G$ is NE.
We must also show that no minor of $G$ is NE.
We first observe that for each choice of $e$, there is another edge $f$ such that $G-e,f$ is planar. That is, $G-e$ is not NE. Also, there's an edge $g$ such that
$(G/e) - g$ is planar, which shows $G/e$ is not NE.

By Theorem~\ref{thmmmnenotclosed}, deleting or contracting further edges continues to give minors of $G$ that are not NE, so long as we do not contract
the unique apex edge in a graph.
Working around this obstacle is not difficult
as we very quickly come to planar minors.
Planarity {\em is} closed under taking minors and a planar graph is
not NE.
\end{proof}

A key step in the classification is the observation that $ab$ is not an edge of $G$.

\begin{lemma}
\label{lemAB_NOT_EDGE}
If $G$ is MMNE, $\K(G) = 2$ with cut set $\{a, b\}$, and $\delta(G) \geq 3$,
  then $ab$ is not an edge in $G$.
\end{lemma}

\begin{proof}
Let $G$ be as described. Let $G - a,b = G_1 \sqcup G_2$ and let $G_i'$ be the induced subgraph of $G$ on the vertices $V(G_i) \cup \{a,b \}$.
For a contradiction, suppose that $ab$ is an edge in $G$. There are
three
cases to consider depending on which of $G'_1$ and $G'_2$ is planar.
If both are planar, then $G$ is
the union of two planar graphs that share an edge and therefore planar.
This contradicts $G$ being MMNE.

Next suppose exactly one of $G'_1$ and $G'_2$ is planar, say $G'_1$.
If $e \in \E(G'_2)$ is an edge other than $ab$, then $G'_2-e$ must be nonplanar.
For otherwise, $G-e$, the union of two planar graphs, $G'_1$ and $G'_2-e$
along $ab$, is planar contradicting $G$ NE. If $G'_2 -ab$ is also nonplanar, then
$G'_2$ is a proper subgraph that is NE,
which contradicts $G$ minor minimal. So, $G'_2 -ab$ is planar.

This means that $G-ab$ is the union of the planar $G'_1-ab$ and the planar $G'_2-ab$, joined at two vertices.
However since $G$ is NE, $G-ab$ is nonplanar, so it has a subgraph homeomorphic to $K_5$ or $K_{3,3}$.
Using Lemma~\ref{lempath}, we know that the subgraph must use only a path through one of $G'_1$, $G'_2$, and nothing else in that component.
This means that one of $G_i^*$ is an edge away from containing a K-subgraph,
where $G_i^*$ denotes $G'_i - ab$. Since $G'_1$ is planar, it must be
$G_2^*$ that contains a subdivision of $K_5$ or $K_{3,3}$ with an edge removed.
Thus, $G'_2$ has a subgraph homeomorphic to $K_5$ or $K_{3,3}$ that uses the
edge $ab$.

Replace $G_1^*$ by the path of Lemma~\ref{lempath} to form a subgraph
$H$ of $G$.  We claim that $H$ is NE. Indeed, deleting $e \in \E(G_2^*)$
leaves $H-e$ 
with the nonplanar subgraph $G'_2-e$. Deleting $ab$ or an
edge in the $G_1^*$ path leaves an $a$-$b$-path that
completes a K-subgraph in $G_2^*$. Since $G$ is minor minimal, $G$ must be
$H$. However, $H$ has at least one degree two vertex, contradicting $\delta(G) \geq 3$.

Finally, we have the case where
$G'_1$ and $G'_2$ are both nonplanar.
Here there are 3 subcases to consider depending on
which of $G_1^* = G'_1 - ab$ and $G_2^* = G'_2 -ab$ is planar.

Suppose first that both $G_1^*$ and $G_2^*$ are planar.
In this case, each of $G'_1$ and $G'_2$ has a K-subgraph that contains $ab$.
It follows that one of the graphs of Theorem~\ref{thmMMNEK21} is a proper minor
of $G$, contradicting $G$'s minor minimality.

In the subcase where both $G_1^*$ and $G_2^*$ are nonplanar,  let $e$ be the apex of $G-ab$. Since the only edge common to $G_1^*$ and $G_2^*$ is $ab$, $e$ is in exactly one of $G_1^*$ and $G_2^*$.
Whichever it isn't in will constitute a nonplanar subgraph of $G-ab,e$, which is a contradiction.

Finally, assume exactly one of $G_1^*$ and $G_2^*$ is planar, say
$G_1^*$.  As above, $G_1^*$ planar and $G'_1$ not implies $G'_1$ contains a
K-subgraph including $ab$ as an edge. On the other hand,
since $G_2^*$ is nonplanar, it has a K-subgraph $H$.
Let $M = G'_1 \cup H$ and, for a contradiction, suppose that $M$ is a proper minor.
Then $M$ must have an apex. However, if we remove an edge $e$ from $G'_1$, then $H$ remains, meaning $M-e$ is nonplanar.
If we remove $e$ from $H$ (which shares no edges with $G'_1$ since it is a subgraph of $G_2^*$),
then $G'_1$ remains, meaning $M-e$ is still nonplanar.
Therefore, no matter what edge we remove from $M$, we can't make it planar and $M$ is NE.
However, $M$ is a minor of $G$, so this contradicts $G$ being MMNE.
Therefore, $H$ is not a proper minor of $G_2^*$, so $G_2^*$ is a subdivision
of $K_5$ or $K_{3,3}$. A similar argument (replace $H$ by $K_5$ or $K_{3,3}$)
shows, in fact, $G_2^*$ {\em is} $K_5$ or $K_{3,3}$ and not just a subdivision.
However, since $ab$ is not an edge of $G_2^*$, then $G_2^*$ must be $K_{3,3}$.

Thus $G_2^* = K_{3,3}$ and $G'_1$ contains a subdivision of $K_{3,3}$
or $K_5$ 
that includes $ab$ as an edge. This means $G$ includes one
of the graphs of Theorem~\ref{thmMMNEK21} as a proper minor and is not minor
minimal.

This completes the last subcase of the last case and shows that $ab$ is not an
edge of $G$.
\end{proof}

For $G$ of connectivity two with cut set $\{a,b\}$, we have $G -a,b = G_1 \sqcup G_2$. We will use $G'_i$ to denote the induced subgraph on $\V(G_i) \cup \{a,b\}$.

\begin{lemma} \label{lemMMNE_NP_SUB}%
	If $G$ is MMNE, $\K(G) = 2$, and $G'_1$ and $G'_2$ are both nonplanar, then $G'_1 = G'_2 = K_{3,3}$.
\end{lemma}

\begin{proof}
Let $G$ be as described.
For the sake of contradiction, suppose
that, without loss of generality, $G'_1$ is nonplanar but not $K_{3,3}$.
Then either $G'_1 = K_5$ or $G'_1$ has a nonplanar proper minor.
If $G'_1 = K_5$, then $ab$ is an edge in $G$, which contradicts Lemma~\ref{lemAB_NOT_EDGE}.
If $G'_1$ has a nonplanar proper minor, $H$, then $H \cup G'_2$ is a proper minor of $G$.
Since there are no edges between $H$ and $G'_2$, the edge apex of $H \cup G'_2$ must be in exactly one of $H$ and $G'_2$.
Whichever one of those two does not contain the edge apex will be a nonplanar subgraph even when the edge is removed.
This contradicts the fact that $G$ is MMNE. Therefore $G'_1 = K_{3,3}$. A similar argument can be made for $G'_2$.
\end{proof}

\begin{lemma}
	If $G$ is MMNE, $\K(G) = 2$, with cut set $\{a,b\}$, $\delta(G) \geq 3$, and
	both
	$G'_1$ and $G'_2$ are planar, then $G_i' \in \{ K_5-e, K_{3,3}-e \}$ with
	$ab$ as the missing edge.
\end{lemma}

\begin{proof}
Let $G$ be as described.
For a contradiction, assume that $G'_1 + ab$ is planar.
Since $G$ is NE, for every $e \in \E(G)$, $G-e$ is nonplanar and, therefore,
has a K-subgraph, $H$.
By Lemma~\ref{lempath} and our assumption that $G'_1 + ab$ is planar,
$H \cap G'_1$ is an $a$-$b$-path. In particular $G'_2 + ab$ is nonplanar.

Note that there are edge disjoint $a$-$b$-paths $P_1$ and $P_2$ in $G'_1$. If not,
say every $a$-$b$-path goes through the edge $e'$.
Then $G - e'$ must be planar as, by Lemma~\ref{lempath}, a  K-subgraph of $G-e'$ would either use a path in $G'_1$, which is not possible as all such paths pass through $e'$, or else use a path in $G'_2$, which is not possible since $G'_1 + ab$ is planar. The contradiction shows there are edge disjoint paths $P_1$ and $P_2$.

This means we can construct a proper minor $M$ of $G$ by adding a triangle on
$ab$. That is, $\V(M) = \V(G'_2) \cup \{c\}$ and  $\E(M) = \E(G'_2) \cup \{ab,bc,ac\}$.
Since $G$ is NE, for any $e \in \E(G'_2)$, $G-e$ is nonplanar with a
K-subgraph that uses only a path in $G'_1$. So, $M-e$ is also nonplanar. On the other hand, if we delete any $e$ in  $\{ab, ac, bc\}$, we are left with a subgraph
of $M-e$ homeomorphic to $G'_2 + ab$. So $M-e$ is again nonplanar. Then
$M$ is a proper NE minor of $G$ contradicting $G$ minor minimal.

We conclude $G'_1 + ab$ is nonplanar. A similar argument shows $G'_2 + ab$
is nonplanar as well. Then $G$ must have one of the NE graphs $G'_1 \ddotcup G'_2$ with $G'_i \in \{K_5 -e, K_{3,3}-e \}$ as a minor. Since $G$ is minor minimal,
$G$ is a graph of this form.
\end{proof}

\begin{lemma}
    If $G$ is MMNE, $\K(G) = 2$, $\delta(G) \geq 3$, $G'_1$ is planar, and $G'_2$ is nonplanar,
	then $G'_1 \in \{ K_5-e, K_{3,3}-e \}$ sharing two vertices and no edges with $G'_2 = K_{3,3}$.
\end{lemma}

\begin{proof}
Let $G$ be as described.
For a contradiction, suppose $G'_1 + ab$ is planar. Then $G'_2+ab$ must
be NE. Indeed, if we delete $ab$, we're left with the nonplanar $G'_2$. Let
$e \in \E(G'_2)$. Since $G$ is NE, $G-e$ is nonplanar and has a K-subgraph $K$.
If $K$ uses at most one of $\{a,b\}$, then K lies entirely in $G'_2$ and avoids $e$. So, $(G'_2+ab)-e$ is nonplanar in this case. On the other hand, if $\{a,b \} \subset \V(K)$, then, by Lemma~\ref{lempath} and since $G'_1 + ab$ is planar, the part of $K$ in $G'_1$ is an $a$-$b$-path. So using edge $ab$ instead, $K$ remains as a
K-subgraph of $(G'_2+ab)-e$, which is again nonplanar. However $G'_2+ab$ NE
contradicts $G$ minor minimal. We conclude $G'_1 + ab$ is nonplanar.

This means $G'_1$ has one of $K_5 -e$ and $K_{3,3}-e$ as a minor with
the missing edge corresponding to $ab$. Replace $G'_1$ by its minor $K_5-e$
or $K_{3,3}-e$, call it $H$, to form $M = H \cup G'_2$, a minor of $G$. We claim $M$ is again NE. Indeed, if we delete $e \in \E(H)$, $G'_2$ shows $M-e$ is nonplanar.
For $e \in \E(G'_2)$, we know $G - e$ has a K-subgraph $K$. If $K$ sees at most one of $a$ and $b$, it must lie entirely in $G'_2$ (since $H$ is planar) and $M-e$ is nonplanar. If $\{a,b\} \subset \V(K)$, then, by Lemma~\ref{lempath}, $K$ is simply a
path on one side of the $2$-cut. If $K$ is a path in $G'_1$, then replace that
by a path in $H$ to recognize $K$ as a subgraph of $M-e$, which is therefore nonplanar. On the other hand, if $K$ is a path in $G'_2$, this path avoids $e$. So, we can use $H$ along with that path to again find a nonplanar subgraph of $M-e$.
Since $G$ is minor minimal, $G = M$ and $G'_1  \in \{K_5-e, K_{3,3}-e \}$ as required.

Now, $G'_2$ being nonplanar has a K-subgraph $K$. Also, there must be an
$a$-$b$-path $P$ in $G'_2$ as otherwise $G$ has connectivity one. Moreover,
both $K$ and $G'_1 \cup P$ are nonplanar, and so they must overlap as otherwise
$G$ has a proper disconnected MMNE minor. This means $P$ passes through $K$ and,
by contracting edges in $P$ if necessary, we can assume $G$ has a minor
with $\{ a, b \} \subset \V(K)$. From this, form the minor $M = G'_1 \cup K$. If $K$
is a subdivision of $K_5$, Then $M$ and hence $G$ has the MMNA graph
$G'_1 \ddotcup (K_5-e)$ as a proper minor, which is a contradiction. So,
$K$ is a subdivision of $K_{3,3}$. After contracting edges, $G$ either has
the MMNA $G'_1 \ddotcup (K_{3,3}-e)$ as a proper minor, which is a contradiction,
or else $G$ has $G'_1 \ddotcup K_{3.3}$ as a minor where $a$ and $b$ are in
the same part of $K_{3,3}$. Since $G$ was minor minimal, we conclude
$G = G'_1 \ddotcup K_{3,3}$. In other words, as required,
$G'_2 = K_{3,3}$ sharing two vertices and no edge with
$G'_1 \in \{K_5-e, K_{3,3}-e\}$.
\end{proof}

\begin{theorem}
	If $G$ is MMNE, $\K(G) = 2$, and $\delta(G) \geq 3$, then $G$ is one
	of the six graphs of Figure~\ref{figMMNEK2}.
\end{theorem}

\begin{proof}
We showed that these six graphs are MMNE in Theorem~\ref{thmMMNEK21}.
The first lemma immediately gives that if $G'_1$ and $G'_2$ are both nonplanar, then they are both $K_{3,3}$.
The second and third lemmas complete the other parts of the proof of the Theorem.
In total, these account for six graphs: one from the first lemma, three from the second, and two from the third.
\end{proof}

The restriction on the minimum degree in the last theorem is necessary. Indeed, there are many MMNE graphs with $\delta(G) = 2$ (meaning $\K(G) \leq 2$). We next show that this is the minimum.

\begin{theorem}
\label{thmmmnemindeg2}%
    The minimum vertex degree in an MMNE graph is at least two.
\end{theorem}
\begin{proof}
    Let $G$ be an MMNE graph.
    Suppose $G$ has some vertex $v$ of degree zero.
    Since $G$ is MMNE and $G-v$ is a proper minor of $G$,
    there is an $e \in \E(G)$ such that $(G-v)-e$ is planar.
    But the addition of a degree zero vertex to a planar graph is planar,
    so $((G-v)-e)+v = G-e$ is planar, which contradicts $G$ being MMNE.

    Next suppose $G$ has some vertex $v$ of degree one.
    Since $G$ is MMNE and $G-v$ is a proper minor of $G$,
    there is an $e \in \E(G)$ such that $(G-v)-e$ is planar.
    But the addition of a degree one vertex to a planar graph is planar
    (since we can shrink down the edge incident to $v$),
    so $((G-v)-e)+v = G-e$ is planar, which contradicts $G$ being MMNE
\end{proof}

Although we cannot completely classify the $\delta(G) = 2$ MMNE graphs, we show that degree two vertices must occur as part of a triangle.

\begin{theorem}
\label{mmnedeg2triangle}
    In an MMNE graph, the neighbors of a degree two vertex
    are themselves neighbors.
\end{theorem}
\begin{proof}
    Let $G$ be an NE graph with a degree two vertex $v$
    with neighbors $a$ and $b$. For a contradiction, suppose
    $ab$ is not an edge of $G$.
    Perhaps $G$ is MMNE so that every proper minor of $G$ is not NE.
    Let $H = G/av$ be the graph that results from contracting
    edge $av$ in $G$.
    Since $G$ is MMNE,
    there must be some edge $e$ in $H$ such that $H-e$ is planar.
    Note that $e$ cannot be the newly formed edge $ab$ in $H$,
    else, since degree one vertices have no impact on the planarity of a graph,
    $G-av$ would also be planar, contradicting $G$ MMNE.
    Consider the graph $G-e$.
    Note that $G-e$ and $H-e$ are homeomorphic,
    so since $H-e$ is planar, $G-e$ is also planar.
    But this contradicts $G$ being MMNE.
\end{proof}

If graph $G$ has a triangle $abc$, a $\TY$ move on $G$ means forming a new graph $G'$ with one additional vertex $v$ (i.e., $V(G') = V(G) \cup \{v \}$) and replacing the edges $ab$, $ac$, and $bc$ with $va$, $vb$, $vc$. So, $G'$ has the same number of edges as $G$ and one additional vertex. In \cite{P}, Pierce shows that $\TY$ often preserves MMNA, as was originally observed by Barsotti in unpublished work. Here we give a similar result for MMNE graphs.

\begin{theorem}
\label{thmmmnety}
    Given an NE graph $G$ with triangle $t$,
    let $G'$ be the result of performing a $\TY$ move on triangle $t$ in $G$,
    and let $v$ be the vertex added in $G'$.
    Graph $G'$ is NE if and only if
    $G'-e_i$ is nonplanar for each $e_i$ incident to $v$.
\end{theorem}

\begin{proof}
    If $G'$ is NE, then $G'-e_i$ is nonplanar
    by definition.
    Conversely suppose that $G'-e_i$ is nonplanar
    for each $e_i$ incident to $v$.
    Perhaps $G'$ is not NE so there is $ e \in \E(G')$
    such that $G'-e$ is planar.
    Note that $e$ cannot be incident to $v$.
    Since $e$ is not part of triangle $t$,
    performing a $\TY$ move on $G-e$ will result in $G'-e$,
    so $\TY$ on $G-e$ is also planar.
    Note that undoing the $\TY$ transform on this graph
    will preserve its planarity.
    However, graph $G-e$ being planar contradicts $G$ being NE.
\end{proof}

We next give an upper bound on the connectivity of MMNE graphs.

\begin{theorem}
If $G$ is MMNE, then $\kappa(G) \leq 5$.
\end{theorem}

\begin{proof}
Suppose $G$ is MMNE with $\kappa(G) \geq 6$ and let $n = | \V(G) |$.
We can assume $n \geq 6$ as $G$ must be nonplanar and the
only nonplanar graph with five or fewer vertices is $K_5$, which is not MMNE.
Since $\kappa(G) \geq 6$, the minimum degree of $G$ is at least six and a lower
bound on $| \E(G) |$ is $6n/2 = 3n$. Now since $G$ is MMNE,
there exists two edges $e$ and $f$ such that $G - e,f$ is a planar graph with at least $3n-2$ edges,
However,
a planar graph on $n$ vertices can have no more than $3n-6$ edges, the number of edges in a planar triangulation. The contradiction shows there is no MMNE graph
with $\kappa(G) \geq 6$.
\end{proof}

Finally, we observe a nice connection between MMNE and MMNA graphs.

\begin{theorem}
If $G$ is MMNE, then $G$ is MMNA or apex.
\end{theorem}

\begin{proof}
Suppose $G$ is MMNE and NA. We will argue that $G$ is in fact MMNA. For this, let $H$ be a proper minor. Since $G$ is MMNE, then $H$ is edge apex. This means either $H$ is already planar, or else there's an edge $e$ such that $H-e$ is planar.
In the latter case, if $v$ is a vertex of $e$, then $H-v$ is again planar. This shows
that $H$ is apex, as required.
\end{proof}

\subsection{Results of Computer Searches}
In addition to the results above,
we have found other examples of MMNE and MMNC graphs
through brute-force computer searches.
The algorithms underlying the searches are fairly straightforward.
First we generate a list of all the graphs that we are going to search
using the gtools that are available with the
\texttt{nauty and Traces} graph theory software \cite{MP}.
    Specifically, we use the gtools \texttt{geng} and \texttt{planarg} to produce
    all connected, nonplanar graphs of minimum vertex degree 
at least two
    that either have
    fewer than $20$ edges or that have fewer than $10$ vertices.
    The commands used to generate these graphs in \texttt{bash} are the following:

\begin{verbatim}
    $ for i in {6..9}; do
        geng -c -d2 ${i} | planarg -v > ${i}v.txt
        done
    $ for i in {10..16}; do
        geng -c -d2 ${i} 0:17 | planarg -v > ${i}v,(0-17)e.txt
        geng -c -d2 ${i} 18   | planarg -v > ${i}v,(18)e.txt
        geng -c -d2 ${i} 19   | planarg -v > ${i}v,(19)e.txt
        done
\end{verbatim}

    This brute force search was carried out on a standard laptop computer
    with 4GB of memory and an Intel Core i3-350M 2.266GHz processor.
    The graphs to be searched were split among many different files
    so that the search could be run in more manageable segments
    and so that we didn't overflow the laptop's memory.
    We chose to limit our search to graphs with fewer than $20$ edges
    or fewer than $10$ vertices due to time constraints.
    There are a total of $158\,505$ connected, nonplanar graphs
    that have $9$ vertices and a minimum vertex degree of 
at least two.
    Searching these graphs took about 
five hours.
    Since there are $9\,229\,423$ 
such
    graphs on $10$ vertices, 
    searching these would take more than 
ten days.
    Similarly it took about 
three
    days to search all $7\,753\,990$ 
    connected,
    nonplanar graphs that have $19$ edges and a minimum vertex degree of 
at least two,
    so searching all $44\,858\,715$ 
    similar graphs 
    on $20$ edges is not feasible.


    Next we reformat these graphs in each file produced
    to be read into Wolfram Mathematica. Then we use Mathematica functions 
    to iterate over this list of graphs one file at a time
    and pull out any that are found to be either MMNE or MMNC.
    The code in Mathematica was run on a single Mathematica kernel
    (no attempt was made to parallelize the search in Mathematica).
An overview of the method of testing if a graph $G$ is MMNE is as follows,
and an analogous method is used to test if a graph is MMNC:
\begin{enumerate} \itemindent 0pt
  \item For each $e \in E(G)$, if $G-e$ is planar return false.
  \item Build all the simple minors of $G$
        (the graphs in $\{G-e, G/e \mid e \in E(G)\}$)
        and remove any duplicates (under isomorphism).
        If for any of these graphs there is no edge $f$
        such that $G-f$ is planar, return false.
  \item \label{enum:sieve}
        Take $S = \{G\} \cup \{G-e \mid e \in E(G)\}$.
        While $S \neq \varnothing$:
        \begin{enumerate}
          \item Reset $S$ to the result of contracting
                each edge of each graph in $S$.
          \item Remove all planar graphs and duplicate
                graphs from $S$.
          \item If there exists $G \in S$ such that
                $G-e$ is nonplanar for each $e \in E(G)$
                then return false.
        \end{enumerate}
  \item Return true.
\end{enumerate}

We need step (\ref{enum:sieve}) explicitly because
both of the properties edge apex and contraction apex
are \emph{not} closed under taking graph 
minors as shown in 
Theorems \ref{thmmmnencex} and \ref{thmmmncncex}.

In addition to the $12$ MMNE graphs that have been considered
in this section, the brute-force search has found
$15$ more examples of MMNE graphs (listed in 
Appendix~\ref{app:mmne}).
Notable graphs in this list are $K_{4,3}$, $K_6-e$,
the rook's graph on 
nine vertices, and some examples
of MMNE graphs with degree two vertices.
The brute-force search also found new examples of MMNC
graphs in addition to the six
graphs considered in this section.
In particular, the computer demonstrated that the 
six MMNE graphs of connectivity two in
Figure~\ref{figMMNEK2} 
are also MMNC.
Along with these graphs there are $69$ other
MMNC graphs on $19$ or fewer edges or $9$ or fewer vertices.
Appendix~\ref{app:mmnc} is an abridged list of these graphs (those on $17$ or fewer edges or $9$ or fewer vertices).

Beyond a simple brute-force search, we also conducted
a more intelligent graph search using the knowledge that
performing $\TY$ and $\YT$ moves on a graph 
has the potential
to preserve the NE or NC property of that graph, see Theorem~\ref{thmmmnety}.
The idea is that the $\TY$ or $\YT$ families of an MMNE or MMNC graph
may contain new MMNE or MMNC graphs.
The details of the methodology of this search
as well as the Mathematica code can be found in \cite{P}.
In total, we have found $55$ MMNE graphs and $82$ MMNC graphs,
and we suspect that there are many more of each.
Tables \ref{tblMMNE} and \ref{tblMMNC} below give a classification
of the MMNE and MMNC graphs we have found organized by graph size.

\begin{table}[h]
\begin{center}
\begin{footnotesize}
\begin{tabular}{r||c|c|c|c|c|c|c|c|c|c c}
    Graph Size ($|E(G)|$) &$\leq$11&     12&     13&     14&     15&     16&     17&     18&     19&     20&\textbf{\ldots}\\\cline{1-11}%
    Number of MMNE Graphs &       0&      1&      0&      2&      0&      2&      3&     11&      6&$\geq$2&\textbf{\ldots}%
                                                                                                            \vspace{5pt}\\
          \textbf{\ldots} &      21&     22&     23&     24&     25&     26&     27&     28&     29&     30&\\\cline{2-11}%
          \textbf{\ldots} &$\geq$13&$\geq$7&$\geq$4&$\geq$2&$\geq$0&$\geq$0&$\geq$1&$\geq$0&$\geq$0&$\geq$1&\\%
\end{tabular}
\end{footnotesize}
\end{center}
\vspace{5pt}
\caption{The number of MMNE graphs we have found grouped by size.
         Note that this is a complete classification based on
         graph size up to and including size $19$.}
\label{tblMMNE}%
\end{table}

\begin{table}[h]
\begin{center}
\begin{footnotesize}
\begin{tabular}{r||c|c|c|c|c|c|c|c|c|c}
    Graph Size ($|E(G)|$) &$\leq$11&12&13&14&15&16&17&18&19&20      \\ \hline
    Number of MMNC Graphs &0       &1 &0 &0 &1 &6 &14&32&25&$\geq$3 \\
\end{tabular}
\end{footnotesize}
\end{center}
\vspace{5pt}
\caption{The number of MMNC graphs we have found grouped by size.
         Note that this is a complete classification based on
         graph size with the exception of size $20$.}
\label{tblMMNC}%
\end{table}

\section*{Acknowledgements}
This material is based upon work supported by the National Science Foundation under Grant Number 1156612.
We received additional support through a Research and Creativity Award from
the Provost's office at CSU, Chico as well as Math Summer Research Internships
from the Math Department. We thank Ramin Naimi and Bojan Mohar for helpful conversations.

\appendix
\section{Edge lists of graphs found through computer searches}
\subsection{MMNE Graphs}
\label{app:mmne}

The following $15$ MMNE graphs are the result of a computer search
conducted on the set of graphs that have $19$ or fewer edges
or $9$ or fewer vertices, and that all have a minimum vertex degree of 
at least two.
These graphs, together with 
eleven
other graphs considered
explicitly in the 
paper (i.e., all but $K_5 \sqcup K_5$, which has order 10 and size 20) make 
up all 
26
MMNE graphs on
$19$ or fewer edges or on $9$ or fewer 
vertices.
(Note that Table~\ref{tblMMNE} gives 25 graphs of size 19 or less. 
Adding the graph $K_5 \dotcup K_5$, of order 9 and size 20, is what brings the 
total to 26.)
\vspace{10pt}
\begin{itemize} \itemsep 8pt \itemindent -5pt
  \item[] $\{(1,8),(1,9),(2,4),(2,7),(2,8),(3,6),(3,7),(3,8),(4,5),(4,6),\\
             (4,8),(5,6),(5,7),(5,9),(6,7),(6,9),(7,9),(8,9)\}$
  \item[] $\{(1,6),(1,7),(2,5),(2,7),(3,7),(3,8),(3,9),(4,5),(4,6),(4,8),\\
             (4,9),(5,7),(5,8),(5,9),(6,7),(6,8),(6,9),(8,9)\}$
  \item[] $\{(1,8),(1,9),(2,6),(2,7),(2,9),(3,5),(3,7),(3,9),(4,5),(4,6),\\
             (4,9),(5,6),(5,7),(5,8),(6,7),(6,8),(7,8),(8,9)\}$
  \item[] $\{(1,8),(1,9),(2,7),(2,10),(3,6),(3,8),(3,10),(4,6),(4,7),(4,9),\\
             (5,6),(5,7),(5,8),(6,9),(6,10),(7,8),(7,10),(8,9),(9,10)\}$
  \item[] $\{(1,9),(1,10),(2,7),(2,8),(2,10),(3,7),(3,8),(3,9),(4,6),(4,8),\\
             (4,10),(5,6),(5,7),(5,9),(6,7),(6,8),(7,10),(8,9),(9,10)\}$
  \item[] $\{(1,6),(1,9),(2,7),(2,8),(3,6),(3,7),(3,10),(4,5),(4,6),(4,7),\\
             (4,10),(5,8),(5,9),(5,10),(6,9),(7,8),(8,9),(8,10),(9,10)\}$
  \item[] $\{(1,8),(1,10),(2,4),(2,8),(2,9),(3,4),(3,5),(3,9),(4,5),(4,6),\\
             (5,7),(5,10),(6,7),(6,8),(6,9),(7,9),(7,10),(8,10),(9,10)\}$
  \item[] $\{(1,6),(1,7),(1,9),(2,7),(2,8),(2,9),(3,6),(3,8),(3,9),(4,5),\\
             (4,8),(4,9),(5,6),(5,7),(5,9),(6,8),(7,8)\}$
  \item[] $\{(1,7),(1,8),(1,9),(2,6),(2,8),(2,9),(3,6),(3,7),(3,9),(4,6),\\
             (4,7),(4,8),(5,6),(5,7),(5,8),(5,9)\}$
  \item[] $\{(1,6),(1,7),(1,8),(2,5),(2,7),(2,8),(3,4),(3,7),(3,8),(4,5),\\
             (4,6),(4,7),(4,8),(5,6),(5,7),(5,8),(6,7),(6,8)\}$
  \item[] $\{(1,6),(1,7),(1,9),(2,5),(2,7),(2,8),(3,7),(3,8),(3,9),(4,5),\\
             (4,6),(4,8),(4,9),(5,7),(5,9),(6,7),(6,8),(8,9)\}$
  \item[] $\{(1,4),(1,7),(1,8),(2,3),(2,7),(2,8),(3,5),(3,6),(4,5),(4,6),\\
             (5,7),(5,8),(6,7),(6,8)\}$
  \item[] $\{(1,5),(1,6),(1,7),(2,5),(2,6),(2,7),(3,5),(3,6),(3,7),(4,5),\\
             (4,6),(4,7)\}$
  \item[] $\{(1,6),(1,7),(1,8),(1,9),(2,4),(2,5),(2,8),(2,9),(3,4),(3,5),\\
             (3,6),(3,7),(4,7),(4,9),(5,6),(5,8),(6,9),(7,8)\}$
  \item[] $\{(1,3),(1,4),(1,5),(1,6),(2,3),(2,4),(2,5),(2,6),(3,4),(3,5),\\
             (3,6),(4,5),(4,6),(5,6)\}$
\end{itemize}

\subsection{MMNC Graphs}
\label{app:mmnc}

The following $22$ MMNC graphs are the result of a computer search
conducted on the set of graphs that have $17$ or fewer edges
or $9$ or fewer vertices, and that all have a minimum vertex degree of 
at least two.
\begin{itemize} \itemsep 8pt \itemindent -5pt
  \item[] $\{(1,9),(1,12),(2,8),(2,11),(3,6),(3,7),(4,5),(4,10),(5,11),(5,12),\\
             (6,9),(6,11),(7,8),(7,12),(8,10),(9,10)\}$
  \item[] $\{(1,6),(1,10),(2,5),(2,9),(3,4),(3,6),(3,8),(4,5),(4,7),(5,10),\\
             (6,9),(7,9),(7,11),(8,10),(8,11),(9,11),(10,11)\}$
  \item[] $\{(1,6),(1,10),(2,7),(2,8),(2,9),(3,6),(3,8),(3,9),(4,7),(4,9),\\
             (4,10),(5,7),(5,8),(5,10),(6,7),(8,10),(9,10)\}$
  \item[] $\{(1,9),(1,10),(2,3),(2,6),(2,7),(3,4),(3,5),(4,7),(4,10),(5,6),\\
             (5,9),(6,8),(6,10),(7,8),(7,9),(8,9),(8,10)\}$
  \item[] $\{(1,9),(1,11),(2,9),(2,10),(3,4),(3,6),(3,11),(4,5),(4,10),(5,8),\\
             (5,9),(6,7),(6,9),(7,10),(7,11),(8,10),(8,11)\}$
  \item[] $\{(1,9),(1,11),(2,9),(2,10),(3,5),(3,6),(3,7),(4,5),(4,6),(4,9),\\
             (5,11),(6,10),(7,8),(7,9),(8,10),(8,11),(10,11)\}$
  \item[] $\{(1,4),(1,11),(2,6),(2,9),(3,5),(3,6),(3,7),(4,5),(4,9),(5,10),\\
             (6,11),(7,9),(7,10),(8,9),(8,10),(8,11),(10,11)\}$
  \item[] $\{(1,9),(1,11),(2,4),(2,5),(2,6),(3,5),(3,6),(3,7),(4,8),(4,9),\\
             (5,11),(6,10),(7,9),(7,10),(8,10),(8,11),(10,11)\}$
  \item[] $\{(1,10),(1,11),(2,3),(2,7),(2,9),(3,6),(3,8),(4,5),(4,9),(4,10),\\
             (5,8),(5,11),(6,7),(6,11),(7,10),(8,10),(9,11)\}$
  \item[] $\{(1,8),(1,9),(2,6),(2,12),(3,5),(3,11),(4,11),(4,12),(5,7),(5,9),\\
             (6,7),(6,8),(7,10),(8,11),(9,12),(10,11),(10,12)\}$
  \item[] $\{(1,9),(1,11),(2,5),(2,12),(3,4),(3,12),(4,8),(4,9),(5,7),(5,9),\\
             (6,7),(6,8),(6,11),(7,10),(8,10),(10,12),(11,12)\}$
  \item[] $\{(1,4),(1,8),(1,9),(2,3),(2,8),(2,9),(3,4),(3,6),(3,9),(4,5),\\
             (4,8),(5,6),(5,7),(5,9),(6,7),(6,8),(7,8),(7,9)\}$
  \item[] $\{(1,4),(1,8),(1,9),(2,4),(2,7),(2,9),(3,4),(3,6),(3,9),(5,6),\\
             (5,7),(5,8),(5,9),(6,7),(6,8),(7,8)\}$
  \item[] $\{(1,5),(1,6),(1,8),(2,3),(2,4),(2,7),(3,6),(3,10),(4,5),(4,10),\\
             (5,9),(6,9),(7,9),(7,10),(8,9),(8,10)\}$
  \item[] $\{(1,5),(1,6),(1,8),(2,3),(2,4),(2,7),(3,6),(3,10),(4,5),(4,9),\\
             (5,10),(6,9),(7,9),(7,10),(8,9),(8,10)\}$
  \item[] $\{(1,2),(1,9),(1,10),(2,7),(2,8),(3,8),(3,9),(3,10),(4,7),(4,9),\\
             (4,10),(5,7),(5,8),(5,10),(6,7),(6,8),(6,9)\}$
  \item[] $\{(1,2),(1,4),(1,10),(2,3),(2,9),(3,4),(3,7),(4,8),(5,7),(5,8),\\
             (5,10),(6,7),(6,8),(6,9),(7,10),(8,9),(9,10)\}$
  \item[] $\{(1,5),(1,6),(1,7),(2,5),(2,6),(2,7),(3,5),(3,6),(3,7),(4,5),\\
             (4,6),(4,7)\}$
  \item[] $\{(1,2),(1,4),(1,7),(1,9),(2,3),(2,6),(2,8),(3,5),(3,6),(3,9),\\
             (4,5),(4,7),(4,8),(5,8),(5,9),(6,8),(6,9),(7,8),(7,9)\}$
  \item[] $\{(1,6),(1,7),(1,8),(1,9),(2,4),(2,5),(2,8),(2,9),(3,4),(3,5),\\
             (3,6),(3,7),(4,7),(4,9),(5,6),(5,8),(6,9),(7,8)\}$
  \item[] $\{(1,5),(1,6),(1,7),(1,8),(2,3),(2,4),(2,7),(2,8),(3,4),(3,6),\\
             (3,8),(4,5),(4,8),(5,6),(5,7),(6,7)\}$
  \item[] $\{(1,2),(1,3),(1,4),(1,5),(1,6),(2,3),(2,4),(2,5),(2,6),(3,4),\\
             (3,5),(3,6),(4,5),(4,6),(5,6)\}$
\end{itemize}

\end{document}